\documentclass{amsart}

\newtheorem{theorem}{Theorem}[section]
\newtheorem{lemma}[theorem]{Lemma}
\newtheorem{proposition}[theorem]{Proposition}

\theoremstyle{definition}
\newtheorem{definition}[theorem]{Definition}
\newtheorem{example}[theorem]{Example}

\theoremstyle{remark}
\newtheorem{remark}[theorem]{Remark}

\numberwithin{equation}{section}

\begin{document}

\title{ON EXTENSIONS OF TYPICAL GROUP ACTIONS}

\author{Oleg\,N.~ Ageev
 }
\address{Department of Mathematics and Mechanics, Lomonosov Moscow State
University, Leninskiye Gory, Main Building, GSP-1, 119991 Moscow,
Russia } \email{ageev@mx.bmstu.ru}
\thanks{The author was supported in part by
 the Programme of Support of Leading Scientific Schools of the RF
(grant no. NSh-5998.2012.1)}


\date{December 12. 2012.}



\begin{abstract}
For every countable abelian group $G$ we find the set of all its
subgroups $H$ ($H\leq G$) such that a typical measure-preserving
$H$-action on a standard atomless probability space $(X,\mathcal{F}
, \mu)$ can be extended to a free measure-preserving $G$-action on
$(X,\mathcal{F} , \mu)$. The description of all such pairs $H\leq G$
was made in purely group terms, in the language of the dual
$\widehat{G}$, and $G$-actions with discrete spectrum. As an
application, we answer a question when a typical $H$-action can be
extended to a $G$-action with some dynamic property, or to a
$G$-action at all. In particular, we offer first examples of pairs
$H\leq G$ satisfying both $G$ is countable abelian, and a typical
$H$-action is not embeddable in a $G$-action.
\end{abstract}

\maketitle

\section*{}
\section{Introduction}

By a transformation $T$ we mean an invertible measure-preserving map
defined on a non-atomic standard Borel probability space
$(X,\mathcal{F} , \mu)$. Iterations of this map that is sometimes
called an automorphism define an action of $\mathbb{Z}$, thus
forming a subgroup of the group of all transformations of
$(X,\mathcal{F} , \mu)$. The $spectral$ properties of $T$ are those
of the induced ( Koopman's) unitary operator on
 $L^2( \mu)$ defined by
\[
\widehat{T}:L^2( \mu)\rightarrow L^2( \mu);\quad
\widehat{T}f(x)=f(Tx).
\]
Transformations and, a bit more generally, group actions (i.e.
 group representations by transformations) are main objects
  to study in modern
ergodic theory. Investigations describing roughly what is changed if
we go to actions of larger groups or back are one of the steadily
well-developing aspects of ergodic theory. In this paper, we take a
global view and, rather than study specific group actions, we are
interested in the spaces of all group actions.

 \begin{definition} We say that groups are
 \textbf{weakly isomorphic} if they are isomorphic to some subgroups
each other.
 \end{definition}
 It is easy to check that this is an equivalence relation, and
 two groups $G_+$ and $G_-$ are weakly isomorphic if and only if
 there exist two groups $G^*_\pm \geq G_\mp$
  which are isomorphic to
 $G_\pm$ respectively. Let us remind that a group is called
 $cohopfian$ if it is not isomorphic to any proper subgroup, and
 $bounded$ if there exists an upper bound of the orders of its elements.
 It is easy to see that $G$ is cohopfian if and only if it is not
 weakly isomorphic to any proper subgroup, i.e. the set
 of its subgroups  that are weakly isomorphic to $G$ is trivial.
 In particular, if $G$ is cohopfian then
  a weak isomorphism of $H$ and
 $G$ actually implies an isomorphism of $H$ and $G$. However, there
 exist a lot of weakly isomorphic and not isomorphic pairs
  even in the class of (infinite) bounded countable abelian groups. It clearly
  follows
  from an explicit description of weakly isomorphic pairs ( see Section
  6).

We say that a typical group action has some property (or the
property is said to be typical) if the set of elements satisfying
the property contains a dense $G_{\delta}$ subset. Following [13], a
property $B$ is called \textbf{dynamic} if it is invariant with
respect to any measure-preserving isomorphism $\varphi$ (i.e.
$B=\varphi^{-1}B \varphi$) and forms a Baire set (i.e. an almost
open set $B=U\Delta M$, where $U$ is open and $M$ is meager). By the
well-known topological $0-1$ law (see [13], [14] ), which applies to
the standard space of group actions $\Omega_G$ for any countable
group $G$, every dynamic property is meager or typical.

 The complete classification of countable abelian
group pairs $H\leq G$ under the condition for a typical $H$-action
to admit an extension to a free $G$-action comes from the following
main theorem.

 \begin{theorem} Let $G$ be any countable abelian group, $H$
 its subgroup. Suppose $H$ is not an infinite bounded group; then
 a typical $H$-action can be extended to a free $G$-action.
 Suppose $H$ is an infinite bounded group; then a typical $H$-action
 can be extended to a free $G$-action if and only if $G$ is weakly
  isomorphic to $H$.
 \end{theorem}

 The choice of a dynamic property " to be free" was made there
 because, in particular, any free $G$-action $T$ can not be reduced
 to an action of a "smaller" group, i.e. the map $T: G\rightarrow
 \Omega$ is an isomorphism. Besides, the set $F$ of all free $G$-actions
 is \textbf{characteristic} among all dynamic properties for the restriction map
$\pi_H :\Omega_G\rightarrow \Omega_H$, where $\pi_H (T)=T|_H$.
Namely,  a typical $H$-action can be extended to a free $G$-action,
equivalently, $\pi_H (F)$ is a typical set, if and only if maps
$\pi^{\pm 1}_H$ send every second category/typical set onto a second
category/typical set only (see Subsect. 6.1). This means that, for a
typical $H$-action, an extension to a free $G$-action implies an
extension to a $G$-action satisfying any particular typical dynamic
property.

The subject we treat in this paper was indirectly initiated by King
who proved that a typical transformation admits at least one root of
any fixed order (see [24]). It can be viewed as a typical
$n\mathbb{Z}$-action can be extended to a $\mathbb{Z}$-action. The
whole $\mathbb{Z}$-actions are clearly free if we extend  free
$n\mathbb{Z}$-actions. In answer to Rudolph-del Junco's and King's
questions (see [19], [24]), it was proved by Ageev (see [2]) that a
typical transformation is not prime and has infinitely many roots of
any order. Actually, it was an easy corollary  of the fact that a
typical transformation is embeddable in a free $\mathbb{Z}\oplus
G$-action for any finite abelian group $G$. Let us also mention
papers [28], [32], proving the embeddability  of a typical
transformation in actions of some non-discrete groups.

Recently Melleray proved Theorem 1.2 for $H$ being
finitely-generated via category-preserving maps and a generalization
of the classical Kuratowski-Ulam theorem (see [26]).

In order to prove Theorem 1.2, we first follow the traditional way
based on the locally dense points technique by extending the proof
of Theorem 1.2 for $H=\mathbb{Z}$ (see [3]) as much as it is
possible. The novelty is we construct a different explicit set of
locally dense points for $\pi_H$ (see Sections 3, 5). So we get one
more proof of the main results of [24], [2] , and [26]. Finally, to
check no free extendability for group pairs $H\leq G$ left, we apply
a group version of the well-known weak-closure theorem that we prove
for every infinite countable abelian group.

As a bonus, we conclude (see Theorem 6.5) that it is closely related
to extensions of ergodic $H$-actions with discrete spectrum. As an
application of the main theorem we then provide the complete
description of all group pairs $H\leq G$, where $G$ is countable
abelian, admitting some extension to a $G$-action for a typical
$H$-action (see Theorems 6.7-9). Moreover, we show that if a typical
$H$-action can be extended to a $G$-action then, in fact, it is
because $\pi_H$ restricted to natural subspaces in $\Omega_G$ send
every second category/typical set to a second category/typical set
(see Theorem 6.7, Remark 6.11). As another application of the main
theorem we describe all countable abelian groups $G$ with
generically monothetic centralizer for a typical $G$-action (see
Theorem 6.12).

 The paper is organized as
follows. In Sect. 2 we collect all needed technical facts. In Sect.
3 we offer a written version of the proof of Theorem 2 announced in
[3] (i.e. Theorem 1.2 for $H=\mathbb{Z}$ only) as  a model case.
Sect. 4 contains two key subtheorems on locally dense points for
$\pi_H$ and on the centralizer of $G$-actions that are of the
independent interest for possible applications, because, in
particular, of its large constructive potential. Theorem 1.2 and its
equivalent versions are proved in Sect. 6 as a consequence of
results from Sect. 3-5. In Sect. 6 we also show  what is changed in
the description of pairs $H\leq G$ if we wish to extend a typical
$H$-action to a $G$-action satisfying any particular typical dynamic
property. Finally, in Sect. 7 we discuss a non-abelian case and
related questions.

\textbf{Acknowledgements}. The author would like to thank  Penn
State University, the George Washington University, the University
of  Sciences and Technology of China (Hefei), California Institute
of Technology, Texas A$\&$M University, the University of North
Carolina, the S. Banach Mathematical Center (Warsaw), the University
of Geneva, and Institute Henri Poincare (Paris) for their support
and hospitality which implicitly led to the completion of this
paper.

\section{Metrics, finiteness, approximations, locally dense points}

Let us identify measure-preserving transformations of
$(X,\mathcal{F} , \mu)$ which are equal for almost every $x$ of $X$.
On the set $\Omega$ of all such transformations the weak (coarse)
topology is defined by the convergency $T_n\rightarrow T$ iff
$\mu(T_n^{-1}A\Delta T^{-1}A)\rightarrow 0$ for any measurable $A$.
It is equivalent to $\mu(T_nA\Delta TA)\rightarrow 0$ for any
measurable $A$. Let $\xi_n\rightarrow\varepsilon$ for some sequence
of finite measurable partitions $\xi_n$, and $\sigma_n$ be the
finite $\sigma$-algebra of all the $\xi_n$-measurable sets. Consider
the metric $d$ on $\Omega$, where
\[
d(T,S)=\sum\frac{1}{n^2}d_n(T,S), \ d_n(T,S)=\max_{A\in\sigma_n}
 \mu (TA\Delta SA).
\]
This is a left-invariant metric defining the same topology. The
$d$-metric is not complete, but $\Omega$ becomes a Polish
topological group with respect to the metric
$d(T,S)+d(T^{-1},S^{-1})$. It is easy to see that if $<T_1,\ldots,
T_n,S_1,\ldots,S_n>$ is an abelian group then
\[
d(T_1\cdots T_n,S_1\cdots S_n)\leq\sum_id(T_i,S_i).\eqno (1)
\]

The set $\Omega_G$ of all the $G$-actions becomes a Polish space
when it is equipped with the weak topology coming from the
convergency
\[
T(n)\rightarrow T \ \mbox{as}\ n\rightarrow \infty\Leftrightarrow
(\forall g \in G)
 T_g(n)\rightarrow T_g\ \mbox{as}\ n\rightarrow \infty.
\]
\subsection{Finiteness}

\begin{definition} Let $H$ be some subgroup of $G$. We say that
a $G$-action $T$ is $\mathbf{H}$-\textbf{finite} if $<T_g:g\in H>$
is a finite subgroup in $\Omega$. We say that a $G$-action is
\textbf{finite} if it is $G$-finite.
\end{definition}
We need the following simple lemma.
\begin{lemma} Every finite $H$-action $P$ can be extended to a
($H$-finite) $G$-action for any countable abelian group $G$ ($H\leq
G$).
\end{lemma}
\begin{proof} The partition of $X$ into $P$-orbits is measurable. We
say that orbits of points $x$ and $y$ are equal if  for any $h\in H$
$P_hx=x$ if and only if $P_hy=y$. It defines an equivalence relation
on $X$ partitioning $X$ into finitely many $P$-invariant measurable
sets of points with equal orbits. It is well enough to extend the
$P$-action on each set separately. Let $X$ consist of equal to each
other orbits only. Then there exists a measurable $P$-invariant
partition $\xi=\{ B_1,\ldots,B_n\}$ of $X$ such that $P$ acts on
$X/\xi$ transitively and $\xi$ separates different elements of each
orbit, i.e. $B_i\ni P_gx\neq x\in B_j$ implies $i\neq j$.

Take any $g_1 \in G\backslash H$. If for all $k\neq 0$ $kg_1\notin
H$, then let $P_{g_1}$ be the identity map. If not, then put $k=\min
l>0: lg_1\in H$. We can split every $B_j$ into $k$ $B_j(i)$
measurable sets of equal measure such that $P_hB_j(i)=B_{m(j,h)}(i)$
for any $h\in H, j, i$. It suffices to define $P_{g_1}$ on, say
$B_1$. Put $P_{g_1}B_1(i)=B_1(i+1)$ if $1\leq i<k$, and
$P_{g_1}B_1(k)=P_hB_1(1)$, where $h=kg_1$. Obviously, we can define
$P_{g_1}|_{B_1(k)}$ such that for any $x\in B_1(1)$
$P_{g_1}^kx=P_hx$. By the commutativity, $P_{g_1}$ can be naturally
defined on the whole $X$. It is clear that $P$ is a well-defined
finite $<H,g_1>$-action.

Iterating the above process we get an $H$-finite $G$-action we need.

\end{proof}

\subsection{Approximations by finite actions}

Let $G$ be any countable abelian group. Then $G=
<g'_1,\ldots,g'_k,\ldots>$. Put $G_k=<g'_1,\ldots,g'_k>$. Consider a
sequence of positive integers $q_n$, where $q_n|q_{n+1},
n=1,\ldots,$ and for every positive integer $k$ there exists $n$
such that $k|q_n$. By $\xi_n$ denote the partition of $X=[0,1)$ into
$q_n$ half-open intervals of equal length .

Let $L_{n,k}$ be the set of all $G_k$-actions $ P$ preserving the
partition $\xi_n$ such that for any $g, j$ $ P_g|_{C_j(n)}$ is just
a \textbf{shift} $Q$, i.e. $Qx=x+\alpha$, $x\in [0,1)$, where
$\alpha$ depends on $g,j$ and $\xi_n=\{C_1(n),\ldots,C_{q_n}(n)\}$ .
Obviously, $ P_g$ is a permutation of $X/\xi_n$ for any $g$ and $ P$
is a finite action. Moreover, $L_{n,k}$ is a finite set, and if $
P_gC_j(n)=C_j(n)$, then  $ P_g|_{C_j(n)}$ is the identity map. By
Lemma 2.2 we can extend each $ P$ (in many ways) to some $G$-action
$ P'$. Fixing a representative $ P'$ to each $ P\in L_{n,k} $, we
get a finite collection, noted $L_{n,k}^*$, in $\Omega_G$. Here each
$P$ belongs to many $L_{n,k}$, but its representatives $P'(n,k)$
might be different to each other.

If $G$ is not a $torsion$ group, then put $L_{n,k}'=\{P\in L_{n,k}:
P \mbox{ acts transitively on}$ $ X/\xi_n \}$, $n,k\in \mathbb{N}$,
and $L_{n,k}'^*=\{P'\in L_{n,k}^*: P'|_{G_k}\in L_{n,k}'\}$. If $G$
is an infinite torsion group, then $L_{n,k}' (L_{n,k}'^*)$ is
defined as before for only change the sequence $q_n$ by
$q'_n=\#G_n$. If $G$ is a finite group, then $L_{n,k}'$ is not
defined.

The aim of this subsection is to prove the following lemma.
\begin{lemma} For any positive integers $n_0, k_0$, the sets
\[
\bigcup_{n>n_0}\bigcup_{k>k_0} L_{n,k}^* \mbox{ and }
\bigcup_{n>n_0}\bigcup_{k>k_0} L_{n,k}'^*
\]
are dense in $\Omega_G$.
\end{lemma}
\begin{proof} It is well known that the set of all free
actions is dense in $\Omega_G$.

 Fix some neighborhood of a free action $T$ in
$\Omega_G$. It contains a cylindric (open) set
 \[ N(\gamma',{\tilde g_1},\ldots,{\tilde g_l}) = \{S\in \Omega_G :
d(S_{\tilde g_i}, T_{\tilde g_i})< \gamma',\ i=1,\ldots,l \},
(\gamma'>0),
\]
where $d$ is a metric on $\Omega$.
 Choose $ k>k_0$
such that both ${\tilde g_i}\in G_{k}$ $i=1,\ldots,l$, and if $G$ is
not a torsion group then $G_k$ contains at least one element of
infinite order. Next it is convenient to represent $G_k$ as the
direct sum of cyclic subgroups. Namely, let
\[
G_{k}=\bigoplus_{i=1}^{r} < g_i>,
\]
where $\mathbf{deg}\  g_i=p_i, i=1,...,r$, $p_i=+\infty$, $i<r_0$,
$p_i$ are primes for $i=r_0,\ldots,r$, $\mathbf{deg}\ g$ is the
order of $g$. It is clear that, for some $\gamma, \gamma">0$ and for
a positive integer $n'$
 \[ N(T, \gamma, g_1,\ldots,g_r,\xi_{n'})\subseteq N(\gamma",
  g_1,\ldots, g_r)\subseteq
 N(\gamma',{\tilde g_1},\ldots,{\tilde g_l}),
\]
here
\[
 N(T, \gamma, g_1,\ldots, g_r,\xi)=\{S\in
\Omega_G :(\forall  D\in \xi ) \ \mu(S_{ g_i}D\Delta T_{ g_i}D)<
\gamma,\ i=1,\ldots,r \}
 \]
  for any finite measurable
partition $\xi$. Take $m>\max_{i\geq r_0}p_i$, consider sets
\[
K_m=\{g\in G_k: g=\sum_ij_ig_i \mbox{ for some }0\leq j_i<p'_i=\min
\{p_i, m\}, i=1,\ldots,r\}.
\]
By the multidimensional Rokhlin lemma, for any $\varepsilon >0$ we
can find a measurable set, say $A$, such that $\bigsqcup_{g\in K_m}
T_gA$ is $X$ up to $\varepsilon$-measure. Therefore, we get a finite
partition $\eta'$ of $X$. Let us slightly change each $T_{g_i}$ by
$T'_{g_i}$, $i<r_0$. Namely, put $T'_{g_i}x=T_{g_i}^{1-m}x$ for any
$x\in T_gA$, $g\in \{g\in K_m :j_i(g)=m-1\}$, and let $T'$ be the
identity $G_k$-action on $Y=X\backslash \bigsqcup_{g\in K_m} T_gA$.
It is clear that $T'$ is a finite $G_k$-action, and for any $i$
$T'^{p'_i}_{g_i}$ is the identity map. Moreover, for sufficiently
small $\varepsilon$ and $m$ large enough any extension of $T'$,
noted $T'^*$, is an element of $ N(T, \gamma/2, g_1,\ldots,
g_r,\xi_{n'})$. Denote $\eta=\eta'\vee\bigvee_{g\in
G_k}T'_g\xi_{n'}$. Since $\eta$ is a finite subpartition of
$\xi_{n'}$, there exists $\delta>0$ such that
\[
N(T'^*, \delta, g_1,\ldots, g_r,\eta)\subseteq N(T'^*, \gamma/2,
g_1,\ldots,g_r,\xi_{n'}).
\]
Let $\{A_1,\ldots, A_s\}=\{D\in \eta: D\subseteq A\}$. Since $\eta$
is a $T'$-invariant partition, $\eta= \{Y, T'_gA_i, i=1,\ldots,s,
g\in K_m\}$. Then every  $\bigsqcup_{g\in K_m} T'_gA_i$ is a
$T'$-invariant set. Since $\xi_n\rightarrow\varepsilon$, for $n>n_0$
large enough we can approximate every $T'_gA_i$ and $Y$ very well by
some $\xi_n$-measurable set $A_i(\xi_n,g)$ ($Y_n$) such that
$\#\{D\in \xi_n: D\subseteq A_i(\xi_n,g)\}$ does not depend on $g$,
$\#K_m|\#\xi_n$, and $\eta_n= \{Y_n, A_i(\xi_n,g), i=1,\ldots,s,
g\in K_m\}$ is a partition of $X$. It is easy to see that we can
find $P\in L(n,k)$ such that for any $i,j, g\in K_m$
$P_{g_i}A_j(\xi_n,g)=A_j(\xi_n,g+g_i)$, and $P_{g_i}^{p'_i}$ is the
identity map, where if $g=\sum_lj_lg_l$ and $j_i=p'_i-1$, then
$A_j(\xi_n,g+g_i)=A_j(\xi_n,\sum_{l\neq i}j_lg_l)$. This means that
any extension $P^*$ of $P$ is an element of $N(T'^*, \delta,
g_1,\ldots, g_r,\eta)$, and $\cup_{n>n_0}\cup_{k>k_0} L_{n,k}^*$ is
dense in $\Omega_G$.

$A priori$, $P$ constructed above is not transitive on $X/\xi_n$.
However, it has a "good" orbital structure. Namely, $X\backslash
Y_n$ consists of $P$-orbits $O_i=\bigsqcup_{g\in K_m}
P_gC_{s_i}(n)$, $C_j(n)\in \xi_n$. Since $\#K_m|\#\{j:
C_j(n)\subseteq Y_n\}$, changing $P$ on $Y_n$ if $Y (Y_n)$ are not
empty sets, we get the same orbits $O_i$, $i=1, \ldots, q'$ on whole
$X$.

If $G$ is not a torsion group, then $p_1=+\infty$, $p'_1=m$, and we
can change $P_{g_1}$ on $B_n=\sqcup_i\sqcup_{g\in
K_m:j_1(g)=m-1}P_gC_{s_i}(n)$ such that $P\xi_n=\xi_n$, $P$ is a
$G_k$-action by permutations of $X/\xi_n$, and  $X=\sqcup_{g\in
K'_m}P_gC_{s_1}(n)$, where $K'_m$ is defined as $K_m$, the only
difference is we change $p'_1$ by $p''_1=q'p'_1$. Since for $m,n$
large enough the measures of sets $Y, Y_n, \sqcup_{g\in
K_m:j_1(g)=m-1}P_gA, B_n$ are sufficiently small, we get $P\in
L'(n,k)$, and $P^*\in N(T, \gamma, g_1,\ldots,g_r,\xi_{n'})$.

If $G$ is an infinite torsion group, then extend $P$ to a
$G_n$-action $P'$ by defining $P'_{g'_{k+1}},\ldots,P'_{g'_n} $
consecutively. Namely, put $t=\min t'>0: t'g'_{k+1}\in G_k$. Let
$P'_{g'_{k+1}}$ be a shift on every $C_j(n)$,
$P'_{g'_{k+1}}\xi_n=\xi_n$, $P'_{g'_{k+1}}C_{s_i}(n)=C_{s_{i+1}}(n)$
if $i\neq 0 \ (\mbox{mod }t) $, and
$P'_{g'_{k+1}}C_{s_i}(n)=P_gC_{s_{i-t+1}}(n)$ if $i|t $, where
$g=tg'_{k+1}\in G_k$. By the commutativity to every $P_g$, $g\in
G_k$ , we can define $P'_{g'_{k+1}}$ uniquely on whole $X$.
Obviously, the $G_{k+1}$-action $T'$ has $\#G_n/\#G_{k+1}$ orbits on
$X/\xi_n$. Iterating this process, we get $P'\in L'(n,n)$, and
$P'^*\in N(T'^*, \delta, g_1,\ldots, g_r,\eta)$ ($Y$ is the empty
set). Lemma 2.3 is proved.

\end{proof}

\subsection{Locally dense points technique}
Suppose $X$ and $Y$ are Polish spaces and $\varphi : X\to Y$
 is  a continuous map. The subset
$C$ of $Y$ is called an $analytic$ set if there exists a Borel
subset $B$ of $X$ such that $\varphi(B)=C$. Next we will essentially
make use of the fact that every analytic set is almost open ( see
[25]).

We denote by $LocDen\varphi$ the set of all $x\in X$, called
$locally$ $dense$ points, such that for any neighborhood $U(x)$ of
$x$, the set $\varphi (U(x))$ is dense in some neighborhood of
$\varphi(x)$. \vspace{3mm}

 \begin{lemma}(see [24], [2])  Let $LocDen\varphi$
 be dense in $X$. Then $\varphi(X)$ is not meager in $Y$.
 Moreover, $\varphi(A)$ is not meager in $Y$ for every non-meager
 $A$ in $X$.
\end{lemma}

\section{Model case}

\begin{example} of a free action $T$ of any countable abelian group
$G$ having an element $g_1$,
 of an infinite order such that $T_{g_1}$
is ergodic and has discrete spectrum.
\end{example}
We will only construct such an action, because the proof  is
obvious. At first fix some representation
\[
     G=\bigcup_{k}G_k,
\]
where $G_k=<g_1,\ldots,g_k>$ is a subgroup generated by
$g_1,\ldots,g_k, k\in \mathbb{N}, g_{k+1}\notin G_k, g_1=g$ . Next
we define a sequence of elements ${\tilde {\alpha}}^{(k)}=({\tilde
{\alpha}_1}^{(k)},\ldots)$
  from $\bar X=\mathbb{R}^\mathbb{N}$ as follows.
\par
Denote, ${\tilde {\alpha}}_i^{(1)}=\beta_i$ $i=1,2,\ldots$,
 where (irrational) $\beta_i$ are rationally independent in the sense
$\sum_{i=1}^{n} r_i \beta_i =0$ $\pmod{1}$ gives $r_i=0$
$i=1,\ldots,n$ for any rational $r_1,\ldots, r_n$ and $n\in
\mathbb{N}$. Fix some sequence of irrationals $\gamma_k$.
\par
If for any $n$, $ng_{k+1}\notin G_k$, then put ${\tilde
\alpha}_{k+1}^{(k+1)}=\gamma_k$ and ${\tilde \alpha}_m^{(k+1)}=0$ as
$m\ne k+1$.
\par
If not, take
\[
n_0=\min_{ng_{k+1}\in G_k \atop n>0} n.
\]
\par
Therefore $n_0g_{k+1}=l_1 g_1+\cdots+l_k g_k$ for some $l_i\in
\mathbb{Z}$. Put
\[
\begin{array}{l}
{\tilde\alpha}^{(k+1)}=\frac{l_1}{n_0}{\tilde\alpha}^{(1)}+
\cdots + \frac{l_k}{n_0}{\tilde\alpha}^{(k)}+e^{(k)},\\
e_{k+1}^{(k)}=\frac{1}{n_0}, \  e_m^{(k)}=0 \mbox{ as }  m\ne k+1,
e^{(k)}=(e_1^{(k)}, e_2^{(k)},\ldots).
\end{array}
\]
Let then ${\alpha}^{(k)}=\pi_{\infty}{\tilde\alpha}^{(k)}$, where
$\pi_{\infty}$  is a natural projection from $\bar X$ onto $ X
=\mathbb{R}/\mathbb{Z} \times \mathbb{R}/\mathbb{Z} \times \ldots$.
It is enough to define for any
\[
g=\sum_{i=1}^{k} m_i g_i
\]
the action $T$  by
\[
\begin{array}{l}
T_g=T^{m_1}_{g_1}\cdots T^{m_k}_{g_k},\\
T_{g_k}x=x+\alpha^{(k)} \mbox{ } \pmod{1}
\end{array}
\]
on $(X, \mu)$, where $\mu$ is the Haar measure on $X$. Here, of
course, $T_g$ does not depend on our choice of the representation
$g$ as
\[
\sum_{i=1}^{k} m_ig_i.
\]
\par
\begin{lemma}
{\it The action $T$ constructed in Example 3.1 is a locally dense
point for a map $\pi_1:\Omega_G\to\Omega$, where
$\pi_1(T)=T_{g_1}$.}
\end{lemma}
\begin{proof} Every ergodic automorphism with discrete spectrum has rank 1
(see[17]), i.e. for $T_{g_1}$ we can choose([6]) a monotonic
sequence of partitions $\xi'_q=\{C_1(q),\ldots, $ $C_{h(q)}(q),
D(q)\}$ such that $T_{g_1} C_i(q)=C_{i+1}(q)$ for $i<h(q)$,
$\mu(D(q))\to 0$ as $q\to\infty$, and $\xi'_q\to\varepsilon$. For
partitions $\xi_q=\{C_1(q),\ldots, C_{h(q)}(q)\}$ we have
$\xi_q\to\varepsilon$. It is more convenient then to consider in
$\Omega$ the $d$ metric associated to $\{\xi_q\}_{q=1}^{\infty}$.
Pick some neighborhood of $T$. It contains a cylindric set
$N(\gamma',g'_1,\ldots,g'_l) = \{S\in \Omega_G :
d(S_{g'_i},T_{g'_i})< \gamma'$ $i=1,\ldots,l \}$. Choose $k$ such
that $g'_i\in G_k$ $i=1,\ldots,l$. We can represent
$G_k=<g_1,\ldots,g_k>$ as the direct sum of the form
\[
G_k=\bigoplus_{i=1}^{r} <\tilde g_i>,
\]
where ${\tilde g_1},\ldots,{\tilde g_{r_0}} (1\le r_0 \le r)$ have
infinite order, $p_i$ $(p_i< \infty)$ is an order of ${\tilde g_i}$
for $ r_0< i\le r$. Moreover,
\[
\begin{array}{l}
g_1=k_0{\tilde g_1}+l_{r_0+1}{\tilde g_{r_0+1}}+\cdots +
l_{r}{\tilde g_{r}}
\end{array}
\]
for some $k_0,l_i\in \mathbb{N}$. It is clear that, for some
$\gamma>0$, if $d(S_{\tilde g_i},T_{\tilde g_i})<\gamma$
$i=1,\ldots,r$, then $S\in N(\gamma',g'_1,\ldots, g'_l)$. Thus it is
good enough to find a $\delta >0$ and a dense set of ${\tilde
T}_{g_1}$ in $U_{\delta}(T_{g_1})$ such that for some ${\tilde
T}_{\tilde g_1},\ldots, {\tilde T}_{\tilde g_{r}},{\tilde
T}_{g_{k+1}}\ldots$
 we have both
\[
\begin{array}{l}
({\tilde T}_{\tilde g_1},\ldots, {\tilde T}_{\tilde g_{r}},{\tilde
T}_{g_{k+1}}\ldots) \mbox{  forms}  \mbox{ a } G-\mbox{action},
\end{array}
 \eqno (2)
\]
and
\[
d({\tilde T}_{\tilde g_i}, T_{\tilde g_i}) <\gamma \ \ i=1,\ldots,
r, \eqno (3)
\]
where $U_\delta(R)=\{ S\in\Omega :d(R,S)< \delta \}$.
\par
Because $T_{g_1}$ has rank 1, we have $C(T_{g_1})=\overline
{\{T^m_{g_1}\}_{-\infty}^{+\infty}}$  in $\Omega$. Denote
$t=p_{r_0+1}\cdots p_{r}$. Therefore take $m_1,\ldots,m_{r}\in
\mathbb{N}$
 such that
\[
\begin{array}{l}
d(T^{m_i}_{g_1},T_{{\tilde g}_i})< \frac{\gamma}{6} \ \
i=1,\ldots,r,
\end{array}
\eqno (4)
\]
\[
\begin{array}{lcr}
d(T^{m_ip_i}_{g_1},E)=d(T^{m_ip_i}_{g_1}, T^{p_i}_{{\tilde g}_i}) <
\frac{\gamma}{9t}  \ \  r_0<i<r ,
\end{array}
\eqno (5)
\]
\[
\begin{array}{l}
  d(T^{m_1k_0}_{g_1} T^{l_{r_0+1}m_{r_0+1}}_{g_1}\ldots
T^{l_{r}m_{r}}_{g_1}, T_{g_1})= \\
d(T^{m_1k_0}_{g_1} T^{l_{r_0+1}m_{r_0+1}}_{g_1}\ldots
T^{l_{r}m_{r}}_{g_1} , T^{k_0}_{{\tilde g}_1} T^{l_{r_0+1}}_{{\tilde
g}_{r_0+1}} \ldots T^{l_{r}}_{{\tilde g}_{r}}) < \frac{\gamma}{9}.
\end{array}
\eqno (6)
\]
\par
Let $L(q)$ be the set of cyclic permutations $\xi_q$, i.e. $L(q)=\{
T_u\in \Omega : T_u \xi_q=\xi_q$ $\&$ $T_u^{h(q)}=E$ $\&$
 $\forall i$  $(T_u^m C_i(q)=C_i(q) \Leftrightarrow m=0 \pmod{h(q)})\}$.
Since $\xi_q\to\varepsilon$, it implies (see [15]) the density of
$\cup_{q>q_0}L(q)$ in $\Omega$ for any $q_0$. Take $q_0$ such that
\[
\sum_{q>q_0}^{\infty} \frac{1}{q^2} d_q(T,S)<\frac{\gamma}{18t} \
\eqno (7)
\]
for any $T,S\in \Omega$. Using (6), choose $\delta >0$ such that
\[
\forall T_k \in\Omega:{} d(T_{g_1},T_k)<\delta \Rightarrow
d(T^{m_i}_{g_1},T^{m_i}_k) < \frac{\gamma}{6} \  \& \eqno (8)
\]
\[
d(T_{g_1}^{m_ip_i},T_k^{m_ip_i})<\frac{\gamma}{9t}, \ \ r_0<i\le r \
\& \eqno (9)
\]
\[
d(T_k^{m_1k_0}T_k^{l_{r_0+1}m_{r_0+1}}\cdots T_k^{l_r m_r}, T_k)<
2\frac{\gamma}{9}. \eqno (10)
\]
For every
\[
T_u\in\bigcup_{q>q_0}L(q)\cap U_{\delta}(T_{g_1}),
\]
let us define ${\tilde T}_{{\tilde g}_1} , \ldots ,{\tilde
T}_{{\tilde g}_{r}}$ such that (2) and (3) hold. (Here $T_u={\tilde
T}_{g_1}$, and the choice of remaining ${\tilde T}_{g_{k+1}},{\tilde
T}_{g_{k+2}},\ldots$ is more or less obvious, because all
$G_k$-actions below are finite).
\par
Consider $\xi_q$, where $q$ is defined by $T_u\in L(q)$ $(q>q_0)$.
Change the numbering of $C_j(q)$ by $C_{(i)}(q)$, where $i$ is
considered modulo $h(q)$ and is found from $T_u^i C_1(q)=C_j(q)$.
Then
\[
T_u C_{(i)}=C_{(i+1)}(q) \ \ (i \ \ \pmod{h(q))}.
\]
Cut every $C_{(i)}(q)$ into $k_0\cdot p_{r_0+1}\cdots p_{r}$
measurable sets $C_{(i)}(i_{r_0},\ldots, i_{r})$ of the equal
measure, where the collection $(i_{r_0},\ldots, i_{r})$ is
considered modulo $(k_0, p_{r_0+1},\ldots, p_{r})$ and
\[
T_u C_{(i_0)} (i_{r_0},\ldots, i_{r}) = C_{(i_0+1)}(i_{r_0},\ldots,
i_{r})
\]
for any $(i_0,i_{r_0},\ldots, i_{r})$ $(i_0\pmod{h(q)})$.
\\
It is clear that we can find ${\tilde T}'_{{\tilde g}_j} \
j=r_0,r_0+1,\ldots, r$ such that
\[
{\tilde T}'_{{\tilde g}_j} C_{(i_0)} (i_{r_0},\ldots,i_{r}) =
C_{(i_0)} (i_{r_0},\ldots, i_{j-1},i_{j}+1,i_{j+1},\ldots,i_{r}),
\]
 $T_u, {\tilde T}'_{{\tilde g}_{r_0}},\ldots,
{\tilde T}'_{{\tilde g}_{r}}$ form an abelian  group, and have order
$h(q),k_0,p_{r_0+1},\ldots,p_{r} $ respectively. Moreover, this
group is just the direct sum of cyclic subgroups generated by them.

From our choice above ${\tilde T}_{{\tilde g}_j}$ are needed to be
very close to $T^{m_j}_u$. For $x\in C_{(i)} (i_{r_0},\ldots,i_{r})$
let
\[
 {\tilde T}_{{\tilde g}_j}=T_u^{m_j}\  \mbox{for} \ 1<j\le r_0,
\]
\[
{\tilde T}_{{\tilde g}_j}x = \left\{
\begin{array}{lc}
T_u^{m_j}{\tilde T}'_{{\tilde g}_j}x & \mbox{if} \ \ i_j\ne p_j-1\\
T_u^{m_j-p_jm_j}{\tilde T}'_{{\tilde g}_j}x  &  \mbox{if} \ \ i_j=
p_j-1
\end{array}
\mbox{for } r_0< j \le r, \right.
\]

\[
{\tilde T}_{{\tilde g}_1}x = \left\{
\begin{array}{lc}
T_u^{m_1}{\tilde T}'_{{\tilde g}_{r_0}}x & \mbox{if} \ \ i_{r_0} \ne
k_0-1 \\ {\tilde T}_{{\tilde g}_{r_0+1}}^{-l_{r_0+1}}\ldots {\tilde
T}_{{\tilde g}_{r}}^{-l_{r}} T_u^{1+m_1-k_0 m_1} {\tilde
T}'_{{\tilde g}_{r_0}}x &
 \ \ \mbox{if} \ \
i_{r_0} = k_0-1 .
\end{array}\right.
\]
It is clear that automorphisms (or permutations of $C_{(i)}
(i_{r_0},\ldots ,i_{r})$) ${\tilde T}_{{\tilde g}_1}, {\tilde
T}_{{\tilde g}_2},\ldots, {\tilde T}_{{\tilde g}_{r}}$ form a
non-free $G_k$-action such that $ T_u={\tilde T}_{{\tilde
g}_1}^{k_0} {\tilde T}_{{\tilde g}_{r_0+1}}^{l_{r_0}+1} \cdots
{\tilde T}_{{\tilde g}_{r}}^{l_r} $. For $1<j\le r_0$, using
(4),(8), we have
\[
d({\tilde T}_{{\tilde g}_j},T_{{\tilde g}_j})=d(T_u^{m_j},
T_{{\tilde g}_j})\le d(T_u^{m_j},T_{g_1}^{m_j})+
d(T_{g_1}^{m_j},T_{{\tilde g}_j})<\frac{\gamma}{3}.
\]
In order to confirm (3) for $r_0<j\le r$, we take an element
$A\in\xi_{q_1}$, where $q_1\le q_0$. The set $A$ consists of a union
of some $C_{(i)}(q)$. Let
\[
H_0(j)=\bigcup_{i,i_{r_0},\ldots,i_{r}} C_{(i)} (i_{r_0},\ldots
,i_{j-1},0,i_{j+1},\ldots,i_{r}).
\]
By definition,
\[
\begin{array}{l}
\mu(T_u^{-m_j}A\triangle {\tilde T}_{{\tilde g}_j}^{-1}A) =
\mu(T_u^{m_j}A\triangle {\tilde T}_{{\tilde g}_j}A) =\\
2\mu(T_u^{m_j}A \backslash {\tilde T}_{{\tilde g}_j}A) =
2\mu((T_u^{m_j}A \bigcap H_0(j))\backslash
({\tilde T}_{{\tilde g}_j}A\bigcap H_0(j)))=\\
2\mu((T_u^{m_j}A \bigcap H_0(j))\backslash (T_u^{m_j-p_jm_j} A
\bigcap H_0(j)))=
\frac{2}{p_j}\mu(A\backslash T_u^{-p_jm_j}A)=\\
\frac{1}{p_j}\mu(T_u^{p_j m_j}A \triangle A).
\end{array}
\]
Therefore,
\[
d_{q_1}(T_u^{m_j},{\tilde T}_{{\tilde g}_j})= \frac{1}{p_j}
d_{q_1}(T_u^{p_jm_j},E)  \mbox{  for }\ q_1\le q_0.
\]
Now if we recall (5), (7), (9), we get
\[
d(T_u^{m_j},{\tilde T}_{{\tilde g}_j})\le \frac{1}{p_j} d (T_u^{p_j
m_j},E)+\frac{\gamma}{9t}\le \frac{\gamma}{3t}, \eqno (11)
\]
Combining this with (4), (8), we obtain
\[
d({\tilde T}_{{\tilde g}_j},T_{{\tilde g}_j})\le d({\tilde
T}_{{\tilde g}_j},T_u^{m_j}) +d(T_u^{m_j},T_{g_1}^{m_j}) +
d(T_{g_1}^{m_j},T_{{\tilde g}_j})\le 2\frac{\gamma}{3}. \eqno (12)
\]
For remaining $j=1$, as above, we have
\[
\begin{array}{l}
\mu(T_u^{m_1}A\triangle {\tilde T}_{{\tilde g}_j}A) =
\frac{2}{k_0}\mu(A\backslash {\tilde T}_{{\tilde
g}_{r_0+1}}^{-l_{r_0+1}}\cdots
{\tilde T}_{{\tilde g}_{r}}^{-l_{r}}T_u^{1-k_0 m_1}A)) =\\
\frac{1}{k_0}\mu(T_u A\triangle T_u^{k_0 m_1} {\tilde T}_{{\tilde
g}_{r_0+1}}^{l_{r_0+1}}\cdots
{\tilde T}_{{\tilde g}_{r}}^{l_{r}}A); \\
d(T_u^{m_1},{\tilde T}_{{\tilde g}_1})\le \frac{1}{k_0}
d(T_u,T_u^{k_0 m_1}{\tilde T}_{{\tilde g}_{r_0+1}}^{l_{r_0+1}}
\cdots {\tilde T}_{{\tilde g}_{r}}^{l_{r}})+\frac{\gamma}{9t}.
\end{array}
\eqno (13)
\]
Taking into account  $(1)$, from (11) we have
\[
d(T_u^{m_js},{\tilde T}_{{\tilde g}_{j}}^{s})\le\frac{s\gamma}{3t}
\]
for $j> r_0$ , $s\in \mathbb{N}$. Thus
\[
d({\tilde T}_{{\tilde g}_{r_0+1}}^{l_{r_0+1}}\cdots {\tilde
T}_{{\tilde g}_{r}}^{l_{r}}, T_u^{l_{r_0+1}m_{r_0+1}} \cdots
T_u^{l_r m_r})\le \frac{\gamma}{3}.
\]
Combining this with (10),(13), we obtain
\[
d(T_u^{m_1},{\tilde T}_{{\tilde g}_{1}})\le \frac{1}{k_0} d(T_u,
T_u^{k_0m_1} T_u^{l_{{r_0+1}m_{r_0+1}}}\cdots
T_u^{l_{r}{m_{r}}})+\frac{\gamma}{9t}+\frac{\gamma}{3}\le
\frac{2\gamma}{3}.
\]
As in (12) we have
\[
d({\tilde T}_{{\tilde g}_{1}}, T_{{\tilde g}_{1}})<\gamma.
\]
Lemma 3.2 is proved.
\end{proof}

\par
\subsection{ Extensions of typical automorphisms to group actions}
\par

\begin{theorem}  Let $G$ be any countable abelian group
having an element $g_1$ of infinity order. Then a typical
automorphism is embeddable as $T_{g_1}$ in a free group action $T$
of $G$.
\end{theorem}
\par
\begin{proof} Fix $G$ and the natural projection $\pi_1:\Omega_G\to\Omega$,
where $\pi_1(T)=T_{g_1}$.
  All free $G$-actions form a dense $G_\delta$ set $F$ in
$\Omega_G$. Therefore the set $B=\pi_1(F)$ is almost open as an
analytic set. Moreover, it is clearly dynamic. By the topological
$0-1$ law, $B$ is then meager or typical.

Besides, It is well know that conjugates to any free action are
dense in $\Omega_G$ for any countable abelian group $G$. Since
$LocDen \pi_1$ is invariant with respect to any conjugate, by Lemma
3.2 $LocDen \pi_1$ is dense in $\Omega_G$.
  From Lemma 2.4 $B$ is then typical and Theorem 3.3 follows.
\end{proof}

\section{Two key subtheorems}

\begin{theorem} Let $T$ be an  action of a countable abelian group $G$
and $H$ be a subgroup of $G$. If $CL\{ T_g: g\in G \}= CL\{ T_h :
h\in H\}$, then $T$ is a locally dense point for the restriction map
$\pi_H :\Omega_G\rightarrow \Omega_H$, where $\pi_H (T)=T|_H$.
\end{theorem}
\begin{proof} Consider a sequence of positive integers $q_n$, where
$q_n|q_{n+1}, n=1,\ldots,$ and for every positive integer $k$ there
exists $n$ such that $k|q_n$. Let $\xi_n$ be the partition of
$X=[0,1)$ into $q_n$ half-open intervals of equal length. It is
clear that the sequence $\xi_n$ is monotonic and $\xi_n\rightarrow
\varepsilon$ as $n\rightarrow\infty$. From now on $d$ is the metric
in $\Omega$ associated to the sequence $\xi_n$.

Fix some neighborhood of $T$ in $\Omega_G$. It contains a cylindric
(open) set
 \[ N(\gamma',g'_1,\ldots,g'_l) = \{S\in \Omega_G :
d(S_{g'_i},T_{g'_i})< \gamma',\ i=1,\ldots,l \}, (\gamma'>0).
\]

 Choose $k_0$
such that $g'_i\in G_{k_0}$ $i=1,\ldots,l$. Next it is convenient to
represent $G_{k_0}=<g_1,\ldots,g_{k_0}>$ as the direct sum of the
form
\[
G_{k_0}=\bigoplus_{i=1}^{r} <\tilde g_i>,
\]
where $\mathbf{deg}\ {\tilde g_i}=p_i, i=1,...,r$, $p_i$ are primes
or $+\infty$. It is clear that, for some $\gamma>0$
 \[ N(\gamma,{\tilde g_1},\ldots,{\tilde g_r})\subseteq N(\gamma',g'_1,\ldots,g'_l) .
\]
All we need to prove is $\pi_H (N(\gamma,{\tilde g_1},\ldots,{\tilde
g_r}))$ is dense in some neighborhood $U\subseteq\Omega_H$ of
$T|_H$. Let us choose then such $U$, a dense subset of $H$-actions
${\tilde T}$ in $U$, admitting extensions to some $G$-actions
${\tilde T}$ from $N(\gamma,{\tilde g_1},\ldots,{\tilde g_r})$.
Denote $H_n$=$G_n\cap H$,
\[
s_i=\min_{m{\tilde g_i}\in <H_n,{\tilde g_1},\ldots,{\tilde
g_{i-1}}> \atop m>0} m,
\]
or $s_i=+\infty$ if for any $m>0$ $m{\tilde g_i}\notin <H_n,{\tilde
g_1},\ldots,{\tilde g_{i-1}}>, i=1,\ldots, r$. If $s_i\neq+\infty$,
then
\[
s_i{\tilde g_i}=h(i)+\sum_{j<i}k_j(i){\tilde g_j},
\]
for some $h(i)\in H_n$. It is clear, that, in general, $s_i, h(i),
k_j(i)$ depend on $n$, but, in fact, $s_i$ are monotonic for any
$i$, therefore they are independent of $n$ for sufficiently large
$n>N_0$, and then we can fix corresponding $h(i), k_j(i)$ for
$s_i\neq+\infty$. Denote $c=1$ if for any $i$ $s_i=+\infty$. If not,
then $c=\max_{i,j}|k_j(i)|$. Put $\lambda=36(3c)^r$, and $I=\{1\leq
i\leq r: s_i\neq+\infty\}$. There exist $h_1,\ldots,h_r\in H$ such
that
\[
d(T_{h_i},T_{\tilde g_i})<\frac{\gamma}{\lambda(1+rc)},\
i=1,\ldots,r.\eqno (14)
\]
We can choose $k'>\max\{k_0,N_0\}$ such that $h_i\in H_{k'}\leq
H_{k'+1}\leq\ldots,\ i=1,\ldots,r$. Besides we can choose a positive
integer $n_0$ such that
\[
\sum_{n>n_0}^{\infty} \frac{1}{n^2} d_n(T,S)<\frac{\gamma}{\lambda}
\ \eqno (15)
\]
for any $T,S\in \Omega$.

 Let $U$ be the (cylindric) set of all
$H$-actions $S$ with the following properties
\[
d(T_{h_i},S_{h_i})<\frac{\gamma}{\lambda},\ i=1,\ldots,r,\eqno (16)
\]
\[
d(T^{s_i}_{h_i},S^{s_i}_{h_i})<\frac{\gamma}{\lambda},\ \mbox{if }
i\in I,\eqno (17)
\]
\[
d(T_{h(i)}T^{k_1(i)}_{h_1}\ldots
T^{k_{i-1}(i)}_{h_{i-1}},S_{h(i)}S^{k_1(i)}_{h_1}\ldots
S^{k_{i-1}(i)}_{h_{i-1}})<\frac{\gamma}{\lambda},\ \mbox{if } i\in
I.\eqno (18)
\]

By $L_{n,k}$ denote the set of all $H_k$-actions ${\tilde T}$
preserving the partition $\xi_n=\{C_1(n),$ $\ldots,C_{q_n}(n)\}$
such that for any $h, j$ ${\tilde T_h}|_{C_j(n)}$ is a shift.
 Fixing
an extension ${\tilde T'}$ to each ${\tilde T}\in L_{n,k} $, we get
a finite subset $L_{n,k}^*$ of $\Omega_H$. By Lemma 2.3,
\[
\bigcup_{n>n_0}\bigcup_{k>k'} L_{n,k}^*\cap U \mbox{ is dense in }U.
\]
 Let us modify
each ${\tilde T'}\in L_{n,k}^*, (n>n_0, k>k')$ to a $G$-action
${\tilde T}$ in the way ${\tilde T}|_{H_k}={\tilde
T'}|_{H_k}$.\footnote{ Not necessarily ${\tilde T}_h={\tilde T'}_h$
if $h\in H\setminus H_k$.} Then
\[
{\tilde T}|_H\in U\Leftrightarrow {\tilde T'}\in U,
\]
and applying Lemma 2.3 again, we get
\[
\bigcup_{n>n_0}\bigcup_{k>k'}\bigcup_{{\tilde T'}\in
L_{n,k}^*}{\tilde T}|_H\cap U \mbox{ is dense in }U.
\]
If, additionally, because of the construction of ${\tilde T}$,
${\tilde T'}\in L_{n,k}^*\cap U (n>n_0, k>k')$ implies that the
corresponding ${\tilde T}\in N(\gamma,{\tilde g_1},\ldots,{\tilde
g_r})$, then Theorem 4.1 is proved.

 Fix ${\tilde T'}\in L_{n,k}^*  (n>n_0, k>k')$, put ${\tilde T}|_{H_k}={\tilde
T'}|_{H_k} $, i.e. at this step ${\tilde T}$ is an element of $
L_{n,k}$. Extend then the $H_k$-action ${\tilde T}$ to
$<G_{k_0},H_k>$-action ${\tilde T}$ by transformations ${\tilde
T}_{\tilde g_1},\ldots,{\tilde T}_{\tilde g_r}$ defined
consecutively as follows.

Denote ${\tilde s}_i=s_i$ if $i\in I$, ${\tilde s}_i=1$ if
$s_i=+\infty$. Cut every $C_j(n)$ into ${\tilde s}_1\cdots{\tilde
s}_r$ half-open intervals $C_j(j_1,\ldots, j_r)$ of the equal
length, where the collection $(j_1,\ldots, j_r)$ is considered
modulo $({\tilde s}_1,\ldots,{\tilde s}_r)$ and
\[
Q C_m (j_1,\ldots, j_r) = C_t(j_1,\ldots, j_r)
\]
for any $(j_1,\ldots, j_r)$, where Q is the shift defined by
$Q(C_m(n))=C_t(n), m,t=1,\ldots,q_n$. Therefore every ${\tilde
T}_h$, $h\in H_k$, sends  $C_m (j_1,\ldots, j_r)$ to some $C_{m'}
(j_1,\ldots, j_r)$, and $m'$ is independent of coordinates
$(j_1,\ldots, j_r)$. It is clear that if $1<s_i\neq+\infty$, then we
can find transformations $T'_{\Delta_i} $ acting as shifts on each
$C_j(j_1,\ldots, j_r)$ such that
\[
T'_{\Delta_i}C_j(j_1,\ldots, j_r)= C_j (j_1,\ldots,
j_{i-1},j_{i}+1,j_{i+1},\ldots,j_r).
\]
Obviously, $T'_{\Delta_i}$ commute to each other and to any ${\tilde
T}_h$, $h\in H_k$.

Let us finally define ${\tilde T}_{{\tilde g}_i}$ keeping in mind
that ${\tilde T}_{{\tilde g}_1},\ldots,{\tilde T}_{{\tilde
g}_{i-1}}$ are well defined. Put
\[
 {\tilde T}_{{\tilde g}_i}={\tilde T}_{h_i}\  \mbox{if} \ s_i=+\infty,
\]
\[
 {\tilde T}_{{\tilde g}_i}={\tilde T}_{h(i)}{\tilde T}_{{\tilde g}_1}^{k_1(i)}\ldots
 {\tilde T}_{{\tilde g}_{i-1}}^{k_{i-1}(i)}\  \mbox{if} \ s_i=1.
\]
If $1<s_i\neq+\infty$, then for $x\in C_j (j_1,\ldots, j_r)$ let
\[
{\tilde T}_{{\tilde g}_i}x = \left\{
\begin{array}{lc}
{\tilde T}_{h_i}T'_{\Delta_i}x & \mbox{if} \ \ j_i\ne s_i-1\\
{\tilde T}_{h(i)}{\tilde T}_{{\tilde g}_{1}}^{k_1(i)}\ldots
 {\tilde T}_{{\tilde g}_{i-1}}^{k_{i-1}(i)}{\tilde T}_{h_i}^{1-s_i}
 T'_{\Delta_i}x  & \mbox{if} \ \ j_i=
 s_i-1.
\end{array}
 \right.
\]

To see why ${\tilde T}_{{\tilde g}_i}$ is a well-defined
transformation for $1<s_i\neq+\infty$, it is convenient to look at
it as a permutation of the sets $C_j(j_1,\ldots, j_r)$. Indeed, it
is a composition of $T'_{\Delta_i}$ and a map which fixes each
element of the partition $\{B_m(i) , m=0,\ldots,s_i-1\}$, where
\[
B_m(i)=\bigcup_{j,j_1,\ldots,j_r} C_j (j_1,\ldots
,j_{i-1},m,j_{i+1},\ldots,j_{r}).
\]
Moreover, restricted to some level $B_m(i)$, this map is an element
of $<{\tilde T}|_{H_k},$ ${\tilde T}_{{\tilde g}_1},\ldots,{\tilde
T}_{{\tilde g}_{i-1}}>$, so it is a permutation of the sets
$C_j(j_1,\ldots, j_r)$ that makes no change in coordinates
$j_i\ldots,j_r$. Thus ${\tilde T}_{{\tilde g}_i}$ does not change
coordinates $j_{i+1},\ldots,j_r$, and does not depend on them.

Along the same line, it is easy to check, that ${\tilde T}_{{\tilde
g}_i}$ commutes to each element of $<{\tilde T}|_{H_k},{\tilde
T}_{{\tilde g}_1},\ldots,{\tilde T}_{{\tilde g}_{i-1}}>$ provided
the later is an abelian group.

Assuming ${\tilde T}$ is a well-defined  $<H_k,{\tilde
g}_1,\ldots,{\tilde g}_{i-1}>$-action, we get that  ${\tilde T}$ is
 well defined  as an $<H_k,{\tilde g}_1,\ldots,{\tilde g}_{i}>$-action
 too. Indeed, any relation $m{\tilde g}_{i}=h_0\in<H_k,{\tilde
g}_1,\ldots,{\tilde g}_{i-1}>$ $(m\neq 0)$ implies $i\in I$,
$s_i|m$, and ${\tilde T}_{{\tilde g}_i}^m={\tilde T}_{h_0}$, because
 of ${\tilde T}_{{\tilde g}_i}^{s_i}={\tilde T}_{h(i)}{\tilde T}_{{\tilde
g}_{1}}^{k_1(i)}\ldots {\tilde T}_{{\tilde g}_{i-1}}^{k_{i-1}(i)}$.
This means that we extended above ${\tilde T}$ to a finite $<
G_{k_0},H_k>$-action. By Lemma 2.2, we can extend the later to some
$G$-action ${\tilde T}$.

In order to show that ${\tilde T}\in N(\gamma,{\tilde
g_1},\ldots,{\tilde g_r})$ if ${\tilde T}'\in U$, let us prove by
the induction that
\[
d({\tilde T}_{{\tilde g}_{i}}, {\tilde
T}_{h_i})<\frac{\gamma}{3(3c)^{r-i}}, \ i=1,\ldots,r.\eqno (19)
\]

Note that if $s_i=+\infty$, then (19) is trivial. Fix $i\in I$,
$s_i\neq 1$, and some positive integer $n'\leq n_0$. Let $A$ be any
$\xi_{n'}$-measurable set. Then $A$ and ${\tilde T}_{h_i}A$ are
$\xi_n$-measurable. By definition,
\[
\begin{array}{l}
\mu({\tilde T}_{h_i}^{-1}A\triangle {\tilde T}_{{\tilde g}_i}^{-1}A)
=\mu({\tilde T}_{h_i}A\triangle {\tilde T}_{{\tilde g}_i}A) =\\
2\mu({\tilde T}_{h_i}A \backslash {\tilde T}_{{\tilde g}_i}A) =
2\mu(({\tilde T}_{h_i}A \bigcap B_0(i))\backslash
({\tilde T}_{{\tilde g}_j}A\bigcap B_0(i)))=\\
2\mu(({\tilde T}_{h_i}A \bigcap B_0(i))\backslash ({\tilde
T}_{h(i)}{\tilde T}_{{\tilde g}_{1}}^{k_1(i)}\ldots
 {\tilde T}_{{\tilde g}_{i-1}}^{k_{i-1}(i)}{\tilde T}_{h_i}^{1-s_i} A \bigcap
B_0(i)))=\\
\frac{2}{s_i}\mu(A\backslash {\tilde T}_{h(i)}{\tilde T}_{{\tilde
g}_{1}}^{k_1(i)}\ldots
 {\tilde T}_{{\tilde g}_{i-1}}^{k_{i-1}(i)}{\tilde T}_{h_i}^{-s_i}A)=\\
\frac{1}{s_i}\mu({\tilde T}_{h_i}^{s_i}A \triangle {\tilde
T}_{h(i)}{\tilde T}_{{\tilde g}_{1}}^{k_1(i)}\ldots
 {\tilde T}_{{\tilde g}_{i-1}}^{k_{i-1}(i)}A).
\end{array}
\]
Therefore,
\[
d_{n'}({\tilde T}_{{\tilde g}_{i}}, {\tilde T}_{h_i})= \frac{1}{s_i}
d_{n'}({\tilde T}_{h_i}^{s_i},{\tilde T}_{h(i)}{\tilde T}_{{\tilde
g}_{1}}^{k_1(i)}\ldots
 {\tilde T}_{{\tilde g}_{i-1}}^{k_{i-1}(i)})  \mbox{ for }\ n'\le n_0.
\]
Now if we recall (15), we get

\[
\begin{array}{l}
d({\tilde T}_{{\tilde g}_{i}}, {\tilde
T}_{h_i})<\frac{1}{s_i}d({\tilde T}_{h_i}^{s_i},{\tilde
T}_{h(i)}{\tilde T}_{{\tilde g}_{1}}^{k_1(i)}\ldots
 {\tilde T}_{{\tilde g}_{i-1}}^{k_{i-1}(i)})+\frac{\gamma}{\lambda}.
\end{array}\eqno (20)
\]
Remark that (20) is also true for $s_i=1$.
 Taking into account the commutativity, we then have

\[
\begin{array}{l}
d({\tilde T}_{h(i)}{\tilde T}_{{\tilde g}_{1}}^{k_1(i)}\ldots
 {\tilde T}_{{\tilde g}_{i-1}}^{k_{i-1}(i)},{\tilde T}_{h(i)}{\tilde T}
 _{h_{1}}^{k_1(i)}\ldots
 {\tilde T}_{h_{i-1}}^{k_{i-1}(i)})\leq \\\sum_{j<i}|k_j(i)|d({\tilde T}_{{\tilde g}_{j}},
  {\tilde
T}_{h_j})\leq\sum_{j<i}cd({\tilde T}_{{\tilde g}_{j}}, {\tilde
T}_{h_j})<\frac{2\gamma}{9(3c)^{r-i}}.
\end{array}\eqno (21)
\]
Since ${\tilde T}'\in U$, by (17-8),  we get

\[
\begin{array}{l}
d({\tilde T}_{h_i}^{s_i},{\tilde T}_{h(i)}{\tilde
T}_{h_{1}}^{k_1(i)}\ldots {\tilde T}_{h_{i-1}}^{k_{i-1}(i)})< d(
T_{h_i}^{s_i}, T_{h(i)}T_{h_{1}}^{k_1(i)}\ldots
T_{h_{i-1}}^{k_{i-1}(i)})+\frac{\gamma}{18(3c)^{r-i}}.
\end{array}\eqno (22)
\]
Besides,
\[
d( T_{{\tilde g}_i}^{s_i},T_{h(i)} T_{{\tilde g}_{1}}^{k_1(i)}\ldots
 T_{{\tilde g}_{i-1}}^{k_{i-1}(i)})=0.
\]
Therefore, applying (14), we get
\[
\begin{array}{l}
d( T_{h_i}^{s_i}, T_{h(i)}T_{h_{1}}^{k_1(i)}\ldots
T_{h_{i-1}}^{k_{i-1}(i)})\leq s_id(T_{{\tilde
g}_{i}},T_{h_i})+\\\sum_{j<i}|k_j(i)|d( T_{{\tilde g}_{j}},
T_{h_j})<\frac{\gamma s_i}{\lambda}.
\end{array}\eqno (23)
\]
Combining (21-3), we obtain
\[
\frac{1}{s_i}d({\tilde T}_{h_i}^{s_i},{\tilde T}_{h(i)}{\tilde
T}_{{\tilde g}_{1}}^{k_1(i)}\ldots
 {\tilde T}_{{\tilde g}_{i-1}}^{k_{i-1}(i)})<\frac{\gamma
 }{\lambda}+\frac{5\gamma}{18(3c)^{r-i}}<\frac{11\gamma}{36(3c)^{r-i}}.
\]
Because of (20), then (19) is proved.

Keeping in mind (19), (16), (14), we get,
\[
d({\tilde T}_{{\tilde g}_{i}},T_{{\tilde g}_{i}})\leq d({\tilde
T}_{{\tilde g}_{i}}, {\tilde T}_{h_i})+d({\tilde T}_{h_i},
 T_{h_i})+d( T_{h_i},T_{{\tilde g}_{i}})<\frac{\gamma
 }{3}+\frac{\gamma
 }{\lambda}+\frac{\gamma
 }{\lambda}<\gamma
\]
for any $i$, and Theorem 4.1 follows.
\end{proof}
Fix some finite measurable partition $\xi=\{C_1,\ldots,C_m\}$ onto
the sets of equal measure (not necessarily $\cup_iC_i=X$, but the
difference will not be anyhow essential in the sequel).

\begin{definition} We say that a $G$-action $P$ is $H,\xi$-finite for
some $H\leq G$ if for any $g\in H$ $P_g\xi=\xi$, $(\exists i_0)
P_gC_{i_0}=C_{i_0}$ implies $P_g$ is the identity, and $P|_H$ is
transitive on $X/\xi=\{1,\ldots,m\}$.
\end{definition}
It is easy to see that every $H,\xi$-finite action $P$ is
$H$-finite, and $<P_g:g\in H>=$ $\oplus_{i=1}^k<P_{g_i}>$. Denote
$p_i=$\textbf{deg} $P_{g_i}$ in $\Omega, i=1,\ldots,k$,
$p=(p_1,\ldots,p_k)$. Mark the element $C_1$, then a map
$\varphi(P_g)=P_gC_1$ is an isomorphism between
$\oplus_{i=1}^k<P_{g_i}>$ and $X/\xi$. Thus we get a parametrization
of $X/\xi$ by elements $\sum c_ig_i, c_i=0,\ldots,p_i-1,
i=1,\ldots,k$.

\begin{definition} We say that a $G$-action $T$ admits a good
approximation by finite actions if there exists a sequence of
$H_n,\xi_n$-finite actions $P(n)$ satisfying both
\[\xi_n\rightarrow\varepsilon, \ \mbox{and}
\]
\[
\omega_n^2\sum_{g:g=\sum c_ig_i(n), |c_i|\leq
2p_i(n)}\mu(T_gC_1(n)\Delta P_g(n)C_1(n))=o(1),
\]
where $\omega_n=\#\xi_n=\prod_ip_i(n)$. \end{definition}
\begin{remark} If a transformation admits a cyclic approximation by
periodic transformations with speed $o(1/n^3)$ (see notation in
[21], [10]), then the corresponding $\mathbb{Z}$-action admits a
good approximation by finite actions.
\end{remark}

Besides, it is easy to see that any action of any finite group can
not admit a good approximation by finite actions.

The proof of the following theorem has an independent interest,
because it provides us by some constructive information about
approximation sequences to elements of centralizers $C\{T_g:g\in
G\}$ for a typical group action $T$.
\begin{theorem}
If an action $T$ of an infinite countable abelian group $G$ admits a
good approximation by finite actions, then
\[
C\{T_g:g\in G\}=CL\{ T_g: g\in G \},\eqno (24)
\]
 and the set of all  $G$-actions  of any infinite countable abelian group $G$
 admitting a good approximation by finite actions forms a typical set.
\end{theorem}
What we really need in the main theorem is just (24) for a typical
group action, and this is well-known for $G=\mathbb{Z}$. Nowadays it
is usually attributed to King (see [23]) exploiting the fact that
rank $1$ transformations are typical. However it was known before
because of a primary proof (see [9]) relying on so called Chacon's
lemma [8]. We do not pretend that the weak-closure theorem was
unknown for any countable abelian group $G$, however we did not find
in the literature any mention about that. We are only aware  of [27]
where it is proved for infinite torsion-free countable abelian
groups. We will follow the primary proof applying the
extension-to-groups technique developed in [1].
\begin{proof} Take a transformation $S\in C\{ T_g: g\in G \}$, assuming $T$
admits a good approximation by  ($H_n$-) finite actions $P(n)$. We
will prove that for any measurable set $A$ and $n$ there exists
$g(n)\in K(n)=\{\sum_ic_ig_i(n): 0\leq c_i< p_i(n)\}$ such that
\[
\lim_{n\rightarrow\infty}\mu(SA\Delta P_{g(n)}(n)A)=0.\eqno (25)
\]
Let us remark first that (25) implies (24). Indeed, for every
$C_j(n)$ we can find $g'=g'(j)\in K(n)$ such that
$C_j(n)=P_{g'}(n)C_1(n)$. Therefore
\[
\mu(T_{g(n)}C_j(n)\Delta P_{g(n)}(n)C_j(n))\leq
\mu(T_{g(n)+g'}C_1(n)\Delta P_{g(n)+g'}(n)C_1(n))+
\]
\[
\mu(T_{g'}C_1(n)\Delta P_{g'}(n)C_1(n)).
\]
From now on we apply that $\rho(A,B)=\mu(A\Delta B)$ is a metric on
the set of all the measurable sets in $X$ if we consider them up to
zero measure. Let $A(\xi_n)$ be a nearest to $A$ $\xi_n$-measurable
set. Obviously,
\[
\rho(\bigcup_jB_j,\bigcup_jD_j)\leq\sum_j\rho(B_j,D_j)
\]
for any collection of measurable sets $B_j, D_j$. Therefore,
\[
\omega_n^2\mu(T_{g(n)}A(\xi_n)\Delta P_{g(n)}(n)A(\xi_n))=o(1),
\]
and
\[
\lim_{n\rightarrow\infty}\mu(SA\Delta T_{g(n)}A)=0.\eqno(26)
\]
In order to get a sequence $g(n)$ that is independent of $A$, we
modify $\xi_n$ to $\xi'_n=\{C'_1(n),\ldots,C'_{\omega_n}(n)\}$,
where $C'_1(n)=C_1(n)\backslash\cup_jT_{g'(j)}^{-1}C_j(n),
C'_j(n)=T_{g'(j)}C'_1(n)$. Choosing a sequence $g(n)$ for $C'_1(k)$,
we get (26) for any $\xi'_k$-measurable set $A$. Since
$\xi'_n\rightarrow \varepsilon$, by the diagonal process, we have a
universal sequence  $g(n)$ and (25) implies (24).

In order to prove (25), denote $C_g(n)=P_g(n)C_1(n)$, $g\in K(n)$.
Then $\xi_n=\{C_g(n):g\in K(n)\}$. We write $A\sim_\varepsilon B$ if
$\rho (A,B)<\varepsilon$,  $1B=B$, and  $0B=\emptyset$ for any set
$B$. Fix $\varepsilon>0$.  For $n$ large enough we have
\[
A\sim_\varepsilon A(\xi_n)=\bigsqcup_{g\in K(n)}a_g(n)C_g(n),
\]
\[
SA\sim_\varepsilon SA(\xi_n)=\bigsqcup_{g\in K(n)}a'_g(n)C_g(n),
\]
where $a_g(n),a'_g(n)\in \{0,1\}$, and
\[
S(A(\xi_n))=\bigsqcup_{g\in K(n)}a_g(n)SC_g(n)\sim_{\delta_1}
\bigcup_{g\in K(n)}a_g(n)T_gSC_1(n),
\]
here $\delta_1\leq\sum_{g\in K(n)}a_g(n)\mu(T_gC_1(n)\Delta
P_g(n)C_1(n))$. Therefore, for $n$ large enough we obtain
\[
SA(\xi_n)\sim_\varepsilon\bigcup_{g\in K(n)}a_g(n)T_gSC_1(n).
\]
Let
\[
M_2(n)=\{x\in C_1(n): T_gx\in P_g(n)C_1(n), g=\sum c_ig_i,
\min_i(2p_i(n)-|c_i|)\geq0\}.
\]
It is clear that
\[
\lim_{n\rightarrow\infty}\omega_n^2\mu(M_2(n)\Delta
C_1(n))=0.\eqno(27)
\]
Denote
\[
B_g(n)=\{x\in C_1(n):Sx\in C_g(n)\},
\]
\[
E_g(n)=T_{-g}SB_g(n), \ D(n)=\sqcup C_g(n).
\]
From now on indexes $g$ in $C_g(n)$ and $a_g(n), a'_g(n)$ will be
considered by "modulo $p(n)=(p_1(n),\ldots)$". Because of (27), for
 $n$ large enough, we have
\[
\max_{g,g_1\in K(n)}\omega_n^2\mu(T_{g_1}C_g(n)\Delta
C_{g_1+g}(n))<\varepsilon, \ \max_{g_1\in
K(n)}\omega_n\mu(T_{g_1}D(n)\Delta D(n))<\varepsilon.
\]
Therefore,
\[
\bigcup_{g\in K(n)}a_g(n)T_{g}(SC_1(n)\backslash
D(n))\sim_\varepsilon
\]
\[
\bigcup_{g\in K(n)}a_g(n)T_{g}SC_1(n)\backslash D(n)\sim_\varepsilon
SA(\xi_n)\backslash D(n)\sim_\varepsilon \emptyset.
\]
\[
SA(\xi_n)\sim_{4\varepsilon}\bigcup_{g\in
K(n)}a_g(n)T_{g}S\bigsqcup_{g_1}B_{g_1}(n)=
\]
\[
\bigcup_{g,g_1\in K(n)}a_g(n)T_{g+g_1}E_{g_1}(n)\sim_{\delta_2(n)}
\bigcup_{g,g_1\in K(n)}a_{g-g_1}(n)T_{g}E_{g_1}(n).\eqno(28)
\]
It is easy to see that
\[
\delta_2(n)<\omega_n^2\max_{g,g'_1}\mu(E_g(n)\Delta T_{g'_1}E_g(n)),
\]
where
 \[
 g'_1\in \{g\in H_n: g=\sum c_ig_i(n), (\forall i)c_i\in
\{0,p_i(n)\}\}.
 \]
 Let
  \[
 x\in \tilde{E}_g(n)=M_2(n)\cap
T_{-g}S(B_g(n)\cap M_2(n))\subseteq E_g(n),
 \]
 and $y=S^{-1}T_gx$. Since $Sy\in T_gM_2(n)$, then $T_{-g'_1}y\in B_g(n)$.
  Hence $T_{-g'_1}x=T_{-g}ST_{-g'_1}y\in E_g(n)$,
 then $\tilde{E}_g(n)\subseteq T_{g'_1}E_g(n)$.
 Since
 \[
 \lim_n\max_gw_n^2\mu(\tilde{E}_g(n)\Delta E_g(n))=0,
 \]
we get
\[
\lim_n\delta_2(n)=0.
\]
Taking into account (28), for $n$ large enough we obtain
\[
SA\sim_\varepsilon\bigsqcup_ga'_g(n)C_g(n)\sim_\varepsilon\bigsqcup_g
T_g(\bigsqcup_{g_1}a_{g-g_1}(n)\tilde{E}_{g_1}(n)).
\]
Since
\[
T_g(\bigsqcup_{g_1}a_{g-g_1}(n)\tilde{E}_{g_1}(n))\subseteq C_g(n),
\]
then for $g\in F(n)=\{g\in K(n):a'_g(n)=1\}$ we get
\[
C_g(n)\sim_{\beta_g}T_g(\bigsqcup_{g_1}a_{g-g_1}(n)\tilde{E}_{g_1}(n))
\sim_{\gamma_g}\bigcup_{g_1}a_{g-g_1}(n)T_gE_{g_1}(n).\eqno(29)
\]
\[
SA\sim_\varepsilon\bigsqcup_{g\in F(n)}C_g(n),\eqno(30)
\]
where
\[
\max\{\sum_{g\in F(n)}\beta_g,\sum_{g\in
F(n)}\gamma_g\}<\varepsilon.
\]
Let us remaind the following combinatorial lemma.
\begin{lemma} (Chacon [8]) Let $||x_{ji}||, j=1,\ldots,k, i=1,\ldots,n$
be a matrix with $0,1$ entries, and $b_j, j=1,\ldots,k$ be
nonnegative numbers, $\sum_jb_j=1$. If there is
$H\subseteq\{1,\ldots,n\}$ such that
\[
\sum_jb_jx_{ji}\geq 1-\eta \ \mbox{for any} \ i\in H,
\]
where $\eta\in (0,1)$, then there exists a $\gamma$
$(1\leq\gamma\leq k)$ such that
\[
\sum_{i\in H}\sum_jb_jx_{ji}x_{\gamma i}\geq \#H(1-2\eta).
\]
\end{lemma}

Fix some $\eta\in (0,1)$. Let
\[
b_{g_1}=\mu(\tilde{E}_{g_1}(n))/\sum_g\mu(\tilde{E}_{g}(n)),
\]
\[
H(n)=\{g\in F(n): \sum_{g_1}a_{g-g_1}(n)b_{g_1}\geq 1-\eta\},
\]
\[
I(n)=F(n)\backslash H(n).
\]
The sets $\tilde{E}_{g}(n)$ are mutually disjoint, therefore from
(29) and (30) we have
\[
\lim_n \mu(\bigsqcup_{g\in I(n)}C_g(n))=0,
\]
\[
SA\sim_\varepsilon\bigsqcup_{g\in
H(n)}T_g\bigsqcup_{g_1}a_{g-g_1}(n)\tilde{E}_{g_1}(n)\sim_\varepsilon
\bigsqcup_{g\in H(n)}C_g(n).\eqno(31)
\]
Applying Lemma 4.6 for $x_{g_1g}=a_{g-g_1}(n)$, we find a $g_0\in
K(n)$ such that
\[
\sum_{g\in H(n)}\sum_{g_1}b_{g_1}x_{g_1g}x_{g_0g}\geq
\#H(n)(1-2\eta), \ \mbox{i.e.}
\]
\[
\sum_{g\in
H(n)}\sum_{g_1}\mu(T_g\tilde{E}_{g_1}(n))a_{g-g_1}(n)a_{g-g_0}(n)\geq
\#H(n)(1-2\eta)\sum_g\mu(\tilde{E}_{g}(n)). \eqno(32)
\]
Take $\eta=\varepsilon/2$. It is clear that
\[
\bigsqcup_{g\in
H(n)}T_g\bigsqcup_{g_1}a_{g-g_1}(n)\tilde{E}_{g_1}(n)\supseteq
\bigsqcup_{g\in
H(n)}a_{g-g_0}(n)T_g\bigsqcup_{g_1}a_{g-g_1}(n)\tilde{E}_{g_1}(n).
\]
Combining with (31-2), for $n$ large enough we obtain
\[
SA\sim_\varepsilon\bigsqcup_{g\in
H(n)}a_{g-g_0}(n)\bigsqcup_{g_1}a_{g-g_1}(n)T_g\tilde{E}_{g_1}(n).
\]
Combining with (29), we get
\[
SA\sim_{2\varepsilon}\bigsqcup_{g\in H(n)}a_{g-g_0}(n)C_g(n)=
\bigsqcup_{g:g+g_0\in H(n)}a_{g}(n)C_{g+g_0}(n)=
\]
\[
P_{g_0}(n)\bigsqcup_{g:g+g_0\in H(n)}a_{g}(n)C_g(n)\subseteq
P_{g_0}(n)\bigsqcup_ga_{g}(n)C_g(n)=P_{g_0}(n)A(\xi_n).
\]
It implies that for $n$ large enough we have
\[
SA\sim_\varepsilon P_{g_0}(n)A,
\]
and (25) is proved.

In order to prove the final part of Theorem 4.5, pick a sequence of
positive numbers $\varepsilon_n\rightarrow 0$ as $n\rightarrow
+\infty$, and $G_k=<g_1,\ldots,g_k>$, where
$G=<g_1,\ldots,g_n,\ldots>$. Consider the set $L_{n,k}'^*$ as in
Lemma 2.3. Obviously, every $P\in L_{n,k}'^*$ is $G_k,\xi_n$-finite.
Fix some presentation $<P_g:g\in G_k>=$
$\oplus_{i=1}^l<P_{g_i(n)}>$, $p_i(n)=$\textbf{deg} $P_{g_i(n)} ,
i=1,\ldots,l$, $C_1(n)$, and $\omega_n=\prod_ip_i(n)$. Denote
\[
U(P)=\{T\in \Omega_G:\omega_n^2\sum_{g:g=\sum c_ig_i(n), |c_i|\leq
2p_i(n)}\mu(T_gC_1(n)\Delta P_gC_1(n))<\varepsilon_n\}.
\]
\[
B_{n_0}=\bigcup_{n\geq n_0}\bigcup_k\bigcup_{P\in L_{n,k}'^*}U(P), \
B=\bigcap_{n_0}B_{n_0}.
\]
Since
\[
\bigcup_{n\geq n_0}\bigcup_kL_{n,k}'^*\subseteq B_{n_0},
\]
applying Lemma 2.3, we get that every $B_{n_0}$ is a dense open set
in $\Omega_G$. Therefore $B$ is a dense $G_\delta$-set and consists
of $G$-actions admitting a good approximation by finite actions.
Theorem 4.5 is proved.

\end{proof}

\section{Examples}

Take some prime number $p$. Let
$\mathbb{C}_{p^\infty}=\mathbb{Q}_p/\mathbb{Z}$, where
$\mathbb{Q}_p$ stands for the additive group of all rational numbers
of the form $i/p^n$, $i, n\in \mathbb{Z}$. Consider a pair
$(\mathbb{C}_{p^\infty}, G)$, where $\mathbb{C}_{p^\infty}$ is a
fixed subgroup of an abelian countable group $G$.

\begin{example} of a free $G$-action $T$  satisfying

\[ CL\{
T_g: g\in G \}= CL\{ T_h : h\in \mathbb{C}_{p^\infty}\}.
\]
\end{example}
Since $\mathbb{C}_{p^\infty}$ is \textbf{divisible}, i.e. for every
its element $a$ and $n$ there is $b\in \mathbb{C}_{p^\infty}$ such
that $nb=a$, $G$ can be represented as
$G=\mathbb{C}_{p^\infty}\oplus G^{*}$ for some $G^{*}$.  It is well
enough to construct such $G$-action $T$ acting by shifts $Q_x$ on
the infinite-dimensional torus
$X=(\mathbb{R}/\mathbb{Z})^\mathbb{N}$ equipped with the Haar
measure, where $Q_xy=x+y$, $x,y\in X$. Fix some splitting of indexes
$\mathbb{N}=I_1\sqcup I_2$, where $\#I_1=\#I_2=\#\mathbb{N}$. Denote
$x'_i(n)=0$ if $i\in I_1$, and $x'_i(n)=1/p^n$ if $i\in I_2$,
$x'(n)=(x'_1(n),\ldots,x'_k(n),\ldots)\in X, T'_{ig_n}=Q^i_{x'(n)},
g_n=1/p^{n}\in\mathbb{C}_{p^\infty}, n,i\in \mathbb{N}$. Then $T'$
is a well-defined $\mathbb{C}_{p^\infty}$-action. Let us modify $T'$
to a $\mathbb{C}_{p^\infty}$-action $T$ we need. Fix a countable
base of open cylindric sets $U_i=\{x\in X: x_{k_j}\in \Delta_j,
j=1,\ldots, m\}$, here $\Delta_j$ are open intervals in
$\mathbb{R}/\mathbb{Z}$. Let $n(i)=\max_jk_j$. Making no loss in the
generality, we can assume that $n(1)<n(2)<\ldots$. Consider the
homomorphism $P$ of $X$ defined by
\[
P(x_1,\ldots,x_k,\ldots)=(px_1,\ldots,px_k,\ldots).
 \]
It is clear that the set of all the pre-images of $x$, i.e.
$\{P^{-n}x : n\in \mathbb{N}\}$ is dense in $X$ for any $x$.
Therefore there exist $m(1)\in \mathbb{N}, \ \tilde{x}(1)\in U_1$
such that $P^{m(1)}\tilde{x}(1)=0$. Put $x_i(m(1))=\tilde{x}_i(1)$
if $i\leq n(1)$, $x_i(m(1))=x'_i(m(1))$ if $i> n(1)$,
$x(m(1))=(x_1(m(1)),\ldots)$. Then $x(m(1))\in U_1$,
$P^{m(1)}x(m(1))=0$, and $x(m(1))=x'(m(1))$ up to finitely many
coordinates. Step by step, by the same argument we can find a
sequence of $x(m(k))\in X$, $m(k)-m(k-1)\in \mathbb{N}$,
$k=2,3,\ldots$ such that $x(m(k))\in U_k$, $P^{m(k)-m(k-1)}x(m(k))$
$=x(m(k-1))$, and $x(m(k))=x'(m(k))$ up to finitely many
coordinates. Denote $T_{g_{m(k)}}=Q_{x(m(k))}$. Then every shift $Q$
on $X$ is contained in $CL\{ T_{g_{m(k)}}: k\in \mathbb{N} \}$, and
 \[
T^{p^{m(k)-m(k-1)}}_{g_{m(k)}}=T_{g_{m(k-1)}}, \mbox{and} \
T^{p^{m(k)}}_{g_{m(k)}} \mbox{is the identity map}.
\]
This means that $T$ can be extended to a
$\mathbb{C}_{p^\infty}$-action $T$.

Let us finally define $T|_{G^*}$. Take any free $G^*$-action
$\tilde{T}$ acting by shifts on $\times_{i\in
I_1}\mathbb{R}/\mathbb{Z}$. Then $T|_{G^*}$ on $X=\times_{i\in
I_1}\mathbb{R}/\mathbb{Z}\times\times_{i\in
I_2}\mathbb{R}/\mathbb{Z}$ is just the diagonal action
$\tilde{T}\times E$, where $E$ is the $G^*$-action on $\times_{i\in
I_2}\mathbb{R}/\mathbb{Z}$ by the identities.

Obviously, $T|_{G^*}$ consists of shifts on $X$. Therefore $T$
extends to a $G$-action by shifts, and
 \[ CL\{ T_g: g\in G \}=CL\{
T_{g_{m(k)}} : k\in \mathbb{N}\}=CL\{ T_h : h\in
\mathbb{C}_{p^\infty}\}.
\]

Let $T_{g'}y=y$ for some $y\in X, g'\in G$. Then $x(g')=0$. By
definition, if $g=i/p^k\in \mathbb{C}_{p^\infty}$, then
$x_j(g)=i/p^k$ for "almost all" $j\in I_2$, and if $g\in G^*$, then
$x_j(g)=0$ for all $j\in I_2$. Therefore, $g'\in G^*$,
$\tilde{T}_{g'}$ is the identity, $g'=0$, and the freeness of $T$
follows.

Let $H^*=\oplus_{i=1}^{\infty} <h_i>$, where \textbf{deg}
$h_i=p_i\rightarrow +\infty$ as $i\rightarrow+\infty$. Consider a
pair $(H^*, G)$, where $H^*$ is a fixed subgroup of an abelian
countable group $G$.

 \begin{example} of a free $G$-action
$T$ satisfying
\[CL\{ T_g: g\in G \}= CL\{ T_h : h\in H^*\}.
\]
\end{example}
\par
It is clear that we can find a sequence $g_i\in G, i\in \mathbb{N}$
such that $G=<g_i:i\in \mathbb{N}>$ and for some sequence $k_i$
$g_{k_i}\in H^*$,  \textbf{deg} $g_{k_i}\rightarrow +\infty$ as $i
\rightarrow +\infty$, and $G_{k_i}=G_{k_{i-1}}\oplus <g_{k_i}>$,
where $G_k=<g_1,\ldots,g_k>$. Indeed, it follows from the fact that
every abelian finitely generated group has the finite torsion part.

Consider some splitting of the indexes
$\mathbb{N}=\sqcup_{i\in\mathbb{N}} I_i$, where for any $i$
$\#I_i=\#\mathbb{N}$. Let $\{\alpha_i\}, i\in\mathbb{N}$ be any
sequence of irrational numbers. Denote
$(t_1,t_2,\ldots)=(1,1,2,1,2,3,1,\ldots)$.

 Assume that at the step
$n$ a free $G_n$-action $T$ acting by shifts $Q_x$ on
$X=(\mathbb{R}/\mathbb{Z})^\mathbb{N}$ was constructed, where for
all $x(g)=(x_1(g),\ldots), g\in G_n$ $x_i(g)=0$ if
$i\in\cup_{m>n}I_m\backslash J_n$ for some finite $J_n\subset
\mathbb{N}$. Let us define a shift $T_{g_{n+1}}$ as follows. If for
any $i$ $n+1\neq k_i$, then denote
\[
m=\min_{sg_{n+1}\in G_n \atop s>0} s,
\]
or $m=+\infty$ if for any $s>0$ $sg_{n+1}\notin G_n$. If
$m=+\infty$, then put $x_i(g_{n+1})=\alpha_{n+1} (\mbox{mod} 1)$ for
$i\in I_{n+1}$, and $x_i(g_{n+1})=0$ for $i\notin I_{n+1}$. If
$m\neq+\infty$ then $mg_{n+1}=g_0\in G_n$, and we put
$x_i(g_{n+1})=1/m (\mbox{mod} 1)$ for $i\in I_{n+1}$, and
$x_i(g_{n+1})=x_i(g_0)/m$ for $i\notin I_{n+1}$. Finally, if
$n+1=k_i$, then put $x_i(g_{n+1})=1/d $ for $i\in
(I_{n+1}\cap\{j>t_i\})\cup\{t_i\}$, and $x_i(g_{n+1})=0$ otherwise,
here \textbf{deg} $g_{k_i}=d$.

It is easy to see that, in any case, we get a well-defined free
$G_{n+1}$-action satisfying   $x_i(g)=0$ if $g\in G_{n+1}$ and
$i\in\cup_{m>n+1}I_m\backslash J_{n+1}$ for some finite
$J_{n+1}\subset \mathbb{N}$. This means that we defined a free
$G$-action $T$. Besides, by the above construction, for any pair of
positive integers $m,n$ there exists $h(m,n)\in H^*$ such that
\textbf{deg} $h(m,n)=d(m,n)>m$, $x_n(h(m,n))=1/d(m,n)$, and
$x_i(h(m,n))=0$ $i=1,\ldots,n-1$. It follows that every shift is
contained in $CL\{T_h: h\in <h(m,n): m,n\in \mathbb{N}> \}$. Thus,
we get
\[CL\{ T_g: g\in G \}= CL\{ T_h : h\in H^*\}.
\]

Consider a pair $(H^{**}, G)$, where $H^{**}$ is a fixed subgroup of
$G=\oplus_{i=1}^\infty\mathbb{Z}/m\mathbb{Z}$ isomorphic to $G$,
$m\in \mathbb{N}$.

\begin{example} of a free $G$-action $T$ satisfying
\[CL\{ T_g: g\in G \}=
CL\{ T_h : h\in H^{**}\}.
\]
\end{example}
\par
Since the subgroup $H^{**}$ is \textbf{servant}, i.e. any equation
$nx=a$ $a\in H^{**}$, $n\in \mathbb{N}$ that is solvable in $G$ is
solvable in $H^{**}$ as well, $G$ can be represented as
$G=H^{**}\oplus G^{*}$ for some $G^{*}$. Construct such $G$-action
$T$ acting by shifts $Q_x$ on
$X=(\mathbb{R}/\mathbb{Z})^\mathbb{N}$, where $Q_xy=x+y$, $x,y\in
X$. Fix some splitting of indexes $\mathbb{N}=I_1\sqcup I_2$, where
$\#I_1=\#I_2=\#\mathbb{N}$. Take any free $G^*$-action $\tilde{T}$
acting by shifts on $\times_{i\in I_1}\mathbb{R}/\mathbb{Z}$. Then
$T|_{G^*}$ on $X=\times_{i\in
I_1}\mathbb{R}/\mathbb{Z}\times\times_{i\in
I_2}\mathbb{R}/\mathbb{Z}$ is the diagonal action $\tilde{T}\times
E$, where $E$ is the $G^*$-action on $\times_{i\in
I_2}\mathbb{R}/\mathbb{Z}$ by the identities. Denote $x'_i(n)=1/m$
if $i=n$, and $x'_i(n)=0$ if $i\neq n$, $x'(n)=(x'_1(n),\ldots)\in
X$. Pick a map $\varphi: I_2\rightarrow G^*$ such that for any $g\in
G^*$ $\#\varphi^{-1}(g)=\#\mathbb{N}$. We can assume that $H^{**}=
\oplus_{i\in I_2}<h_i>$, where \textbf{deg} $h_i=m, i\in I_2$. Put
\[
T_{h_i}=Q_{x(\varphi(i))+x'(i)}=T_{\varphi(i)}Q_{x'(i)}, i\in I_2.
\]
It is easy to see that we got a well-defined $G$-action. Besides, by
definition, for any $g\in G^*$ there exists $i_n\rightarrow
+\infty$, such that $\varphi(i_n)=g$. Therefore
\[
T_{h_{i_n}}=T_gQ_{x'({i_n})}\rightarrow T_g \ \mbox{as} \
n\rightarrow +\infty,
\]
and $T$ is the $G$-action as it is required.

\section{Proof of the main theorem and applications}

 Every bounded abelian group
  $G$  is uniquely represented as the direct sum of finitely many
  $p_i$-groups $G_{p_i}$, where $p_i$ are mutually different prime
  numbers, and $G_p=\{g\in G: \exists k p^kg=0\}$. Moreover, by the classical
 Pr\"{u}fer theorem every
  such $G_{p_i}$ can be represented as the direct sum of cyclic
  subgroups. Calculating multiplicities of summands of equal order,
  we get a finitely supported function, noted
  \[
  m_G: \{p^k: p \mbox{ is prime and }k>0 \}\rightarrow \mathbb{N}\cup
  \{0, +\infty\},
  \]
  which is independent of our choice of representations $G_{p_i}$
  as the direct sums.

It is clear that two bounded countable abelian groups $G_+$ and
$G_-$ are isomorphic iff $m_{G_+}=m_{G_-}$. This implies in
particular that every bounded infinite countable abelian group is
not cohopfian.

Consider $\overline{m}_G$ defined by
\[
\overline{m}_G(p^k)=\sum_{n\geq k} m_G(p^n), p \mbox{ is prime and }
k>0 .
 \]
 Then $H\leq G$
implies $\overline{m}_H\leq \overline{m}_G$. It follows that two
bounded countable abelian groups $G_+$ and $G_-$ are weakly
isomorphic if and only if $\overline{m}_{G_+}=\overline{m}_{G_-}$.

Let us prove the main theorem.

\begin{proof} The case $H$ is finite is trivial. Namely, the
well-known Bernoulli action  $T$ on, say, $\{0,1\}^G$ equipped with
the Haar measure acts freely by shifts. Thus $T|_H$ is free. However
all free $H$-actions are pairwise isomorphic and form a typical set.
This means that a typical $H$-action can be extended to a free
Bernoulli $G$-action.

Let $H$ contain an isomorphic copy of $\mathbb{Z}$. Consider a free
$G$-action $T$ as in Example 3.1. We have
 \[CL\{ T_g: g\in G \}=
CL\{ T_h : h\in H\}= CL\{ T_h : h\in \mathbb{Z}\},
\]
 because the
centralizer of ergodic transformation with discrete spectrum is just
the closure of its powers. By Theorem 4.1, $T$ is a locally-dense
point for $\pi_H$. It is well known that all the conjugate actions
to any fixed free action are dense for any amenable group $G$. This
means that we have a dense set of locally dense points for $\pi_H$.
By the same argument as in the proof of Theorem 3.3, we have that a
typical $H$-action can be extended to a free $G$-action.

Let $H$ be an infinite torsion group. If there exists a subgroup
$H^*=\oplus_{i=1}^{\infty} <h_i>$, where \textbf{deg}
$h_i\rightarrow +\infty$ as $i\rightarrow+\infty$, then we are done
by the same argument based on Example 5.2. Besides, it is easy to
see that $H$ has such subgroup $H^*$ if and only if for any $m,n$
there exists a finitely generated subgroup, noted $H^0$, admitting
some presentation $H^0=\oplus_{i=1}^{N} <h'_i>$, where
$\#\{i:\mathbf{deg} \ h'_i>m\}>n$. This means that no such subgroup
$H^*$ for $H$ implies that for some positive integer $m$ $mH$ has
finite rank. It is well known that an abelian torsion group has
finite rank if and only if it is a direct sum of finitely many
finite cyclic groups and $\mathbb{C}_{p^\infty}$. If there exists
some $\mathbb{C}_{p^\infty}\leq mH\leq H$, then, applying Example
5.1, we get what we claimed. If not, then there exists $m'$ such
that $\#m'H=1$. In this case $H$ can be represented as
$H=\oplus_{j=1}^{s} H_{p_j}$, where $p_j$ are mutually different
prime numbers and $H_p=\{h\in H: (\exists k>0)\textbf{deg} \
h=p^k\}$. Moreover, $H_{p_j}=\oplus_{i}<h_i(j)> $ for any $j$, where
$\mathbf{deg}\  h_i(j)|m'$ for any $i,j$. Denote
\[
d_j=\max k: \#\{i: \mathbf{deg}\  h_i(j)=k\}=\#\mathbb{N},
\]
and $d_j=1$ if $\#H_{p_j}$ is finite, $j=1,\ldots, s$. Let
$d=\prod_jd_j$ and
\[
H_{p_j}^>=\bigoplus_{i:\mathbf{deg}\  h_i(j) >d_j}<h_i(j)>, \
H_{p_j}^\leq=\bigoplus_{i:\mathbf{deg}\  h_i(j) \leq d_j} <h_i(j)>,
\]
\[
H_{p_j}^d=\bigoplus_{i:\mathbf{deg}\  h_i(j) =d_j}<h_i(j)>.
\]
It is clear that $H^>=\oplus_{j}H^>_{p_j}$ is finite and
$H=H^>\oplus H^\leq$, where $H^\leq=\oplus_{j}H^\leq_{p_j}$.

 Let $G$
be any abelian countable group ($H\leq G$). Assume that every $g\in
G$  admits a representation $g=h+h_1$ for some  $h\in H^> $,
$dh_1=0$. Then $G=\oplus_{j=1}^{s} G_{p_j}$ and $H^>_{p_j}$ is
servant in $G_{p_j}$ for any $j$. Therefore $G$ can be represented
as
\[
G=H^>\oplus G^* \ \mbox{for some countable abelian} \ G^*, \#dG^*=1.
\eqno(33)
\]
It is clear that $H=H^>\oplus H'$ for some subgroup $H'\leq G^*$,
and $H'\cong H/H^>\cong H^\leq$. Fix a $G^{**}=\oplus_{i=1}^\infty <
h'_i>$ such that \textbf{deg} $h'_i=d$, and $G^*\leq G^{**}$. Then
we can find a subgroup, noted  $H^d$, of $H'$, which is isomorphic
to $\oplus_{j}H^d_{p_j}\cong \oplus_{i}^\infty
\mathbb{Z}/d\mathbb{Z}$. By Example 5.3, there exists a free $G^*$
(and $G^{**}$) - action $T'$ such that $CL\{ T'_g: g\in G^* \}= CL\{
T_h : h\in H'\}$. Let $T''$ be any free $H^>$-action on a non-atomic
standard Borel probability space $(Y,\mathcal{F}_1 , \nu)$. Then the
cartesian product $T=T'\times T''$ acting on $(X\times
Y,\mathcal{F}\otimes\mathcal{F}_1,\mu\otimes\nu)$ is a free
$G$-action, and $CL\{ T_g: g\in G \}= CL\{ T_h : h\in H\}$.
Therefore $T$ is a locally dense point for $\pi_H$, and a typical
$H$-action can be extended to a free $G$-action.

 If $G$ is still a countable abelian group, $H\leq G$, and is not as
 in (33), then we fix  $g\in G$ that can not be represented as
 $g=h+h^{*}$, $h\in H^> $, $dh^{*}=0$. Besides, by Theorem 4.5, for a
 typical $H$-action $T$, for any $S\in C\{T_h:h\in H\}$, we can find
 $h_i=h_i^>+h_i^\leq$ such that $T_{h_i}\rightarrow S$ as
 $i\rightarrow \infty$. Taking a subsequence, we have
 $T_{h^\leq_i}\rightarrow R=ST^{-1}_{h^>}$, and $R^d$ is the identity.
  Assuming that there
 exists a $G$-action $T$ extending $H$-action $T$, we obtain that
 $T_{h^\leq_i}\rightarrow R=T_gT_{h^>}^{-1}$ for some $h^>\in H^>,
  h_i^\leq\in H^\leq$, $i=1,\ldots$.
 Therefore $R=T_{h^{*}}$ for $h^{*}=g-h^>$, and $dh^{*}\neq 0$.
 Since $R^d$ is the identity, the $G$-action $T$ is not free. Therefore
 a typical $H$-action can not be extended to a free $G$-action.

 Besides, it is easy to check that a countable abelian group $G$
 ($H\leq G$) satisfies (33) if and only if there exists some
 $G^{**}\geq G$ isomorphic to $H$.
Theorem 1.2 is proved.
\end{proof}
\subsection{Applications}
\begin{definition} Let $X$ and $Y$ be Polish spaces.
We say that a (continuous) map $\varphi : X \rightarrow Y$ is
(second/essential) category preserving if the images and the
pre-images of second category sets are of the second category as
well.
\end{definition}
Obviously, if the map $\varphi$ is ess. category preserving, then
$\varphi(X)$ is a typical set, i.e. $Y\setminus\varphi(X)$ is
meager. Besides, every pre-image (not image!) of a meager set is a
meager set for ess. category preserving maps.

Let us remark that for every countable group $G$ the space
$\Omega_G$ has the $weak$ $Rokhlin$ $property$ , i.e. there is a
$G$-action $T$ such that the set of all $G$-actions that are
isomorphic to $T$ is dense in $\Omega_G$ (see [14], [22]). This
implies that for any amenable subgroup $H$ ($H\leq G$) $\pi_H
(\Omega_G)$ is dense in $\Omega_H$\footnote{This is also true for
every non-amenable $H$ as well, because we can correspond to  every
$H$-action $T$  the so-called $co-induced$ $G$-action
$\widetilde{T}$ satisfying $T$ is a factor (a quotient ) of
$\widetilde{T}|_H$ (see [16], [22]). The density follows from a
well-known fact that every factor of any $G$-action $S$ is in the
closure of conjugations of $S$ if $G$ is countable. } .

 We note that if both $\varphi(X)$ is dense in $Y$
  and $\varphi(B)$ is of the second category for
every second category set $B$, then $\varphi$ is ess. category
preserving. Indeed, in this case, for any countable intersection
$\cap_i B_i$ of open dense sets $B_i$, $\varphi(\cap_i B_i)$ ( as an
analytic set) is an open set, say $Y'$, up to a meager set. It is
easy to see that $Y'$ is dense in $Y$. Therefore $\varphi$ maps any
typical set onto a typical set. Fix some set $C\subseteq Y$ of the
second category. Assume $\varphi^{-1} (C)$ is meager; then
$X\setminus\varphi^{-1} (C)$ is typical in $X$, and
$\varphi(X\setminus\varphi^{-1} (C))$ is typical in $Y$. However
$\varphi(X\setminus\varphi^{-1} (C))\cap C=\emptyset$, and we have a
contradiction. Therefore $\varphi^{-1} (C)$ is of the second
category, and $\varphi$ is ess. category preserving.

 It was noticed  in [2], Lemma 1,  that if the map $\varphi$ has
 a dense set of locally dense points, then  $\varphi( B)$ is not
 meager for every non-meager $B$. This implies that $\varphi$ is
 ess. category preserving. Therefore main theorem
has the following application.

\begin{theorem} Let $G$ be any countable abelian group and $H$
 its subgroup. Suppose $H$ is not an infinite bounded group; then
 the restriction map
$\pi_H :\Omega_G\rightarrow \Omega_H$ is ess. category preserving.
 Suppose $H$ is an infinite bounded group; then $\pi_H$ is ess.
 category preserving
  if and only if $G$ is weakly  isomorphic to $H$.
 \end{theorem}

\begin{proof} Indeed, in fact, we proved that the set of locally dense points
is dense for $\pi_H$ and for any pair $(H, G)$ except the case when
$H$ is a bounded subgroup and $G$ is not weakly isomorphic to $H$.
Besides, in remaining, by the main theorem, $\pi_H (F)$ is meager
for every infinite $H$, where $F$ stands for the set of all free
$G$-actions. Therefore $\pi_H$ is not ess. category preserving. And
finally, let $H$ be a finite subgroup of $G$. Assume $\pi_H$ is not
ess. category preserving. Then $\pi_H(A)$ is meager for some
non-meager $A$. It clearly implies that $\pi_H(B)$ is meager for
some non-empty open $B$. Fix  a countable dense set, say $C$, in
$\Omega$. Let $D$ be the union of all $\varphi^{-1}B\varphi$,
$\varphi \in C$. Then $\pi_H(D)$ is meager. Besides, the set, say
$E$, of all $G$-actions that are isomorphic to the Bernoulli
$G$-action $T$ is dense in $\Omega_G$ because $T$ is free. It is
easy to check both $E\subset D$ and $\pi_H(E)$ being  the set of all
free $H$-actions can not be meager. We then have a contradiction and
Theorem 6.2 is proved.
\end{proof}
 \begin{remark} According to [33], a continuous map respects the
genericity if the images and the pre-images of  typical sets are
typical. Relying on the above discussion, it is easy to see that a
continuous map is ess. category preserving if and only if it
respects the genericity.
\end{remark}

The reader can find a bit different look at relations between
alternative conceptions of category preserving maps in [26], [33].

\begin{remark} Let us illustrate how Theorem 6.2 works on a simple
example of a dynamic property to be "not isomorphic to its inverse".
It is well known (see [1]) that a typical $\mathbb{Z}^d$-action is
not isomorphic to its inverse. Let $G$ be any non-torsion countable
abelian group. Take some $g\in G$ of infinite order. By Theorem 6.2
the map $\pi_{<g>}$ is ess. category preserving. Therefore the set
of not isomorphic to its inverse $G$-actions is a typical
set.\footnote{In fact, applying the technique developed in [1], we
can show that a typical $G$-action is not isomorphic to its inverse
for any countable group $G$ except obvious cases when $G$ is finite
or $G=\oplus^\infty\mathbb{Z}/2\mathbb{Z}$. However, the proof is ,
of course, much more complicated.}
\end{remark}

In the sequel, we keep the standard notation $\widehat{G}$ for the
group of characters of $G$, and \textbf{Ann} $H$ for a subgroup
$\{\chi\in \widehat{G}: (\forall h\in H )\chi (h)=0\}$, where $H$ is
a subset of an abelian group $G$.
\begin{theorem} Let $H$ be any infinite subgroup of a countable
 abelian group $G$.
Then the following assertions are equivalent each other:
\begin{description}
 \item[(i)] For a typical $H$-action there is an extension to a free
$G$-action.
  \item [(ii)]  There is a dense subgroup $G_1\leq \widehat{G}$ such
   that $\# G_1\cap \mbox {\textbf{Ann} } H=1$.
  \item [(iii)] There is a free (ergodic) $G$-action $T$ with discrete
  spectrum such that   $T|_H$ is ergodic.
\end{description}
\end{theorem}
Let me leave an ergodic-theoretical proof of Theorem 6.5.
\begin{proof} Suppose first that $H$ is unbounded. Following the
proof of the main theorem, we construct  a free $G$-action $T$
acting by shifts on the infinite-dimensional torus $X$. Moreover,
the set of shifts corresponding to the $H$-subaction forms a dense
set among all the shifts on $X$. It gives (iii). Let $G_1$ be the
discrete part $\Lambda (T)$ of the spectrum  of $T$, i.e. the set of
all eigenvalues of $G$-action $T$. Then $G_1$ is a countable
subgroup of $\widehat{G}$. The freeness of $T$ implies the density
of $G_1$. Besides, if $\lambda \in G_1\cap \mbox {\textbf{Ann} } H$,
then the corresponding eigenfunction $f_\lambda$ is invariant for
$T|_H$. Because of the ergodicity, $f_\lambda$ is a constant
function, and $\lambda$ is the trivial character. It gives (ii).

We can apply the same argument for any infinite bounded $H$ if it is
weakly isomorphic to $G$. The only difference is the set of shifts
corresponding to the $H$-subaction forms a dense set among all the
shifts $T_g, g \in G$ on the infinite-dimensional torus $X$.
Restricting ourselves to an ergodic component for $T|_H$, we get a
$T|_G$-invariant closed subset of $X$, which is just the closure of
a $T|_H$-orbit. Thus we have a free action as it is needed in (iii).

Therefore we get the implications (i)$\Rightarrow$ (iii)
$\Rightarrow$ (ii).

Conversely, since $H$ is infinite, $G_1$ is infinite. Fix an
infinite countable dense subgroup $G_1^{*}\leq G_1$. Consider an
ergodic $G$-action $T$ with discrete spectrum, where $\Lambda
(T)=G_1^{*}$. The density of $G_1^{*}$ implies the freeness of $T$.
Obviously, $\# G_1^{*}\cap \mbox {\textbf{Ann} } H=1$ implies $T|_H$
is an ergodic action with discrete spectrum. Applying the standard
realization of every ergodic $G$-action $S$ with discrete spectrum
as a $G$-action by shifts on the character group $\widehat{\Lambda
(S)}$ of the discrete part $\Lambda (S)\leq \widehat{G}$, it is easy
to get a well-known fact that every transformation that commutes
with $S$ is just a shift on the abelian group $\widehat{\Lambda
(S)}$. Therefore $CL\{ T_g: g\in G \}= CL\{ T_h : h\in H\}$. By
Theorem 4.1, $T$ is a locally-dense point for $\pi_H$. This means
that we have a dense set of locally dense points for $\pi_H$. By the
same argument as in the proof of Theorem 3.3, we have that a typical
$H$-action can be extended to a free $G$-action. Therefore we proved
that (ii)$\Rightarrow$ (iii) $\Rightarrow$ (i).
\end{proof}

\begin{remark} $Apriori$ Theorem 6.2 is not useful if $\pi_H$ is  not
ess. category preserving, because the topological status of
$\pi_H(B)$ can be essentially different for different typical sets
$B$. However the main theorem can be applied implicitly in this case
as well (see Remark 6.11).
\end{remark}
In order to proceed to that, let us describe all the pairs $H\leq G$
satisfying $\pi_H$ is $extremely$ non ess. category preserving, i.e.
$\pi_H(B)$ is meager for every typical set $B$. By the $0-1$ law, it
is well enough to say when $\pi_H(\Omega_G)$ is a typical set.

Notice, first, that for every subgroup $G^{*}\leq G$ the space
$\Omega_{G/G^{*}}$ can be viewed as a closed subset of $\Omega_G$
defined by the embedding $\varphi^{*}:\Omega_{G/G^{*}}\rightarrow
\Omega_G$, where $\varphi^{*}(T)_g=T_{\varphi(g)}, g\in G, T\in
\Omega_{G/G^{*}}$, $\varphi$ is the canonical homomorphism of $G$
onto $G/G^{*}$.

\begin{theorem} Let $G$ be any countable abelian group and $H$
 its subgroup. Suppose $H$ is not an infinite bounded group; then
$\pi_H (\Omega_G)$ is a typical set.
 Suppose $H$ is an infinite bounded group; then $\pi_H(\Omega_G)$ is
 a typical set if and only if there is a subgroup $G^{*}\leq G$ such
  that both $\# G^{*}\cap H=1$ and
 $G/G^{*}$ is weakly  isomorphic to $H$. Moreover, if $H$ is
  an infinite bounded group and there is such a subgroup $G^{*}$,
  then $\pi_H|_{\varphi^{*}(\Omega_{G/G^{*}})}$ is ess. category
  preserving.

 \end{theorem}
\begin{proof} The first part follows from the main theorem. Let
$H$ be an infinite bounded group. "$\Leftarrow$" and the last part
of the theorem is almost obvious. Indeed, $\varphi (H)$ is
isomorphic to $H$ via $\varphi$. Therefore
\[
\pi_H(\varphi^{*}(\Omega_{G/G^{*}}))=\pi_{\varphi
(H)}(\Omega_{G/G^{*}})
 \]
if we naturally identify $\Omega_H$ and $\Omega_{\varphi (H)}$. It
remains to apply Theorem 6.2 for the map $\pi_{\varphi (H)}$.

In order to prove "$\Rightarrow$", we keep the proof of the main
theorem notation. It is easy to see that by Theorem 4.5,  for a
typical $T\in \Omega_H$, if there is a $T$ extending $G$-action $T$
then
\[
\varphi_T(dG)= \varphi_T(dH^>)=\varphi_T(dH),
\]
where $\varphi_T$ is the homomorphism defined by $\varphi_T(g)=T_g,
g\in G$. Since $\pi_H(\Omega_G)$ restricted to the set of free $H$-
actions is a typical set as well, we can choose a free $H$-action
$T'$ admitting an extension to a $G$-action $T'$ such that
$\varphi_{T'}(dG)=\varphi_{T'}(dH)$. Denote $G^*=$ \textbf{Ker}
$\varphi_{T'}$. Obviously, $\# G^{*}\cap H=1$. Thanks to the
homomorphism theorem, the group $G/G^{*}$ is isomorphic to
$\varphi_{T'}(G)$. Besides, $H$ is isomorphic to $\varphi_{T'}(H)$.
It remains to prove that $\varphi_{T'}(G)$ is weakly isomorphic to
$\varphi_{T'}(H)$. However,
 \[
 d\varphi_{T'}(G)=\varphi_{T'}(dG)=\varphi_{T'}(dH)=d\varphi_{T'}(H).
 \]
This completes the proof of Theorem 6.7.
\end{proof}

A more explicit description of all pairs $H\leq G$ with
$\pi_H(\Omega_G)$ is a typical set comes from the following theorem.
\begin{theorem} Let $G$ be any countable abelian group, $H$
 its infinite bounded subgroup. Then $\pi_H(\Omega_G)$ is
 a typical set if and only if $G$ can be represented
 as $ H^{'}\oplus G_1$, where $H\leq H^{'}\leq G$,
  and $H^{'}$ is weakly  isomorphic to $H$.

 \end{theorem}

This theorem implies that a typical $H$-action can be extended to a
$G$-action because there is an extension to a free $H^{'}$-action
and every $H^{'}$-action can be extended to a $G$-action by, for
example, identities on $G_1$.
 \begin{proof} "$\Leftarrow$" is a clear application of Theorem 6.7.
To get "$\Rightarrow$", we employ the proof of the main theorem
notation. Fix a $j$, and a representation  $H_{p_j}=H^>_{p_j}\oplus
H^\leq_{p_j}$.

Remark first that $H^>_{p_j}$ is servant in $G_{p_j}$. If not, then
consider the canonical homomorphism $\varphi$ of $G$ onto $G/G^{*}$,
where $G^{*}$ is as in Theorem 6.7. Therefore $\varphi (H^>_{p_j})$
is not servant in $G/G^{*}$, because $\varphi$ is the group
isomorphism if it is restricted to $H$. However $\varphi
(H^>_{p_j})$ is the direct summand in $G/G^{*}$ because $G/G^{*}$ is
weakly isomorphic to $H$. We get a contradiction.

By the the
 Pr\"{u}fer-Kulikov theorem, for any countable abelian group,
 every its bounded servant subgroup is a direct summand. Thus
$G_{p_j}=H^>_{p_j}\oplus G^{*}_{p_j}$, $H_{p_j}=H^>_{p_j}\oplus
H^{*}_{p_j}$ for some $H^{*}_{p_j}\leq G^{*}_{p_j}$, where
$H^{*}_{p_j}$ is isomorphic to $H^\leq_{p_j}$.

Let us define  $H^{'}_{p_j}$. If $H_{p_j}$ is finite, then we put
 $H^{'}_{p_j}=H_{p_j}=H^>_{p_j}$. Assume that $H_{p_j}$ is
 infinite.
Obviously, $\# d_jH^{*}_{p_j}=1$, and $d_j=p_j^k$ for some positive
integer $k$. Moreover, exploiting $\varphi$ again, we get that
\[
\# H^{*}_{p_j}\cap d_jG^{*}_{p_j}=1.
\]

Denote by $M$ the non-empty set of all subgroups $H^{*}\leq
G^{*}_{p_j}$ such
  that both $\# H^{*}\cap d_jG^{*}_{p_j} =1$ and
 $H^{*}_{p_j}\leq H
 ^{*}$. Consider the restriction to $M$
 of the natural partial order $G_1\leq G_2$ on the set of all subgroups of
 $G$. Fix a maximal element, say
  $ H^{*}_{ max}$. Then $ H^{*}_{ max}$ is servant in $G^{*}_{p_j}$
  (see, for example,
  [12], Theorem 27.7). Obviously,
  \[
  d_jH^{*}_{ max}\subseteq H^{*}_{ max}\cap d_jG^{*}_{p_j}=\{0\},
  \mbox{ and} \ H^{*}_{p_j}\leq H^{*}_{ max}.
  \]
Therefore $ H^{*}_{ max}$ is weakly isomorphic to $H^{*}_{p_j}$.

Put $H^{'}_{p_j}=H^>_{p_j}\oplus H^{*}_{ max}$. It is clear that
$H_{p_j}\leq H^{'}_{p_j}\leq G_{p_j}$, $H^{'}_{p_j}$ is weakly
isomorphic to $H_{p_j}$, and $H^{'}_{p_j}$ is a direct summand in
$G_{p_j}$.

Let us finally define $H^{'}$ as $\oplus_jH^{'}_{p_j}$. It is clear
that $H^{'}$ is weakly isomorphic to $H$, and $H^{'}$ is a direct
summand in the torsion part of $G$. Thus $H^{'}$ is a bounded
servant subgroup of $G$, and Theorem 6.8 is proved.
\end{proof}
Let us mention a dual analog of Theorem 6.8 as well.
\begin{theorem} Let $G$ be any countable
 abelian group $G$, $H$ its infinite subgroup.
Then the following assertions are equivalent:
\begin{description}
 \item[(i)] A typical $H$-action can be extended to a
$G$-action.
  \item [(ii)]  $\widehat{G}$ can be represented as the direct
  sum of two (possibly trivial) closed subgroups $G_1^{*}$ and
  $G_2^{*}$ such that both  $ \mbox {\textbf{Ann} } H\geq G_2^{*}$
    and there is a $G_1^{*}$-dense subgroup $G_1\leq G_1^{*}$
    satisfying  $\# G_1\cap \mbox {\textbf{Ann} } H=1$.

\end{description}
\end{theorem}
This theorem is a natural corollary of Theorems 6.5, 6.8, and we
leave the proof to the reader.
\begin{remark} We omit certain dual versions of Theorem 6.7, or,
say, of a slight modification of Theorem 6.8, where we replace the
conditions on $G$ by $G$ can be represented as $G=G_1\oplus G_2$
such that $\# G_2\cap H=1$ and $G_1$ is weakly isomorphic to $H$.
\end{remark}
\begin{remark} Assume $H$ is an infinite bounded group and
$\pi_H(\Omega_G)$ is
 a typical set. Denote
 \[D=\{G^{*}\leq G :
 \# G^{*}\cap H=1 \ \& \
 G/G^{*} \mbox{ is weakly  isomorphic to }H \}.
 \]
    In fact, proving Theorem 6.7 we showed the following claim.
 \[
 \mbox{For a typical } C, \
 \pi^{-1}_H(C)\subseteq A=\cup_{G^{*}\in D}\varphi^{*}(\Omega_{G/G^{*}}).
 \]
 It implies that the question when $\pi_H(B)$ is a typical set
 is reduced to the same question for the restrictions of $B$ to the
 subspaces $\varphi^{*}(\Omega_{G/G^{*}})$, where $\pi_H$ becomes  ess.
 category preserving.

 In spite of the fact that $D$ may have continuum many elements,
 the set $A$ is always
 nowhere dense if $G$ is not weakly isomorphic to $H$.
  Indeed, the density somewhere implies the density
 everywhere. Therefore, the set of locally dense points for $\pi_H$
 is dense in $\Omega_G$, and $\pi_H$ is ess.
 category preserving.

  Besides, it is easy to check that
 \[
\cup_{G^{*}\in D}
 \varphi^{*}(\Omega_{G/G^{*}})\subseteq B=\cap_g\cup_{h\in
 H^>} \{T\in \Omega_G : T_{dg}=T_{dh}\}.
  \]
We note that, as a rule, the inclusion $A\subseteq B$ is proper.
\end{remark}

We now turn to another certain application of the main theorem.
There is a bit stronger property then to be $monothetic$ for
topological groups. Namely, following [27], a topological group $G$
is called $generically$ $monothetic$ if the set of $g$ that generate
a dense subgroup of G is a dense $G_\delta$-set in $G$.  Note that
in the definition above, the condition that the set be $G_\delta$ is
automatic, so one only has to check that it is dense. Note then that
$G$ is clearly generically monothetic if for any positive integer
$n$ $<ng_0>$  is dense in $G$ for some $g_0\in G$.

There has been considerable interest in the study of centralizers of
typical group actions (see, for example, most new preprints [27],
[31]). One of the new reasons for that, among others, is to
understand all the similarities that come if we replace $\Omega$, as
a target for representations of a group, with other sufficiently
well developed automorphism groups , say with the unitary group of a
separable, infinite-dimensional Hilbert space or the group of
isometries of the universal Urysohn metric space (see, for example,
[27], [34]). Let us remind that the centralizer of a typical
transformation
 is a closed abelian non-locally compact monothetic subgroup of $\Omega$.
\begin{theorem} Let $G$ be any countable abelian group.
Then the centralizer of a typical $G$-action is generically
monothetic if and only if $G$ is unbounded.
\end{theorem}
\begin{proof} It is clearly true for any bounded $G$. Indeed, if $G$
is finite, then the centralizer of every free $G$-action is just the
group of all $G$-extensions under certain topology. Thus, it is not
even abelian. Besides, for any infinite bounded $G$, by Theorem 4.5
every element of the centralizer is of finite order for a typical
$G$-action. This means that the centralizer is not even
(topologically) finitely generated.

Let $G$ contain an isomorphic copy of $\mathbb{Z}$, say  $<g>$.
Combining Theorems 6.2 and 4.5 for $H_n=<ng>$, $n\in \mathbb{N}$, we
easily deduce that
 \[
C\{T_h:h\in G\}=CL \{< T_{ng}>\}, n\in \mathbb{N},
\]
for a typical $G$-action $T$. Therefore, the centralizer is
generically monothetic. It implies, in particular, that the
cetralizer is generically monothetic for a typical
$H\oplus\mathbb{Z}$-action, where $H$ is any unbounded torsion
abelian group. Applying Theorems 6.2 and 4.5 again, we have the same
for a typical $H$-action. This completes the proof.
\end{proof}

\section{Non-abelian case, questions, and comments}

Let $G$ be any countable group. The group $\Omega$ acts on
$\Omega_G$ by conjugations. It is well known that the set  of all
$G$-actions with a dense orbit under conjugations, say $F'=F'_G$,
forms a dense $G_\delta$-set.

Since the set $F$ of free $G$-actions is equal to $F'$ for every
countable amenable group $G$, the set $F$ remains characteristic for
the restriction map $\pi_H$ in the amenable world according to the
following fact.
\begin{proposition} Let $H$ be any subgroup of a countable group $G$.
The following are equivalent:\begin{description}
\item[(i)] The map $\pi_H$ is ess. category preserving.
   \item[(ii)] $\pi_H(F')$ is a typical set.
  \end{description}
\end{proposition}
\begin{proof} The implication (i)$\Rightarrow$ (ii)
is obvious. To get (ii) $\Rightarrow$ (i), we first argue as in
[27], Proposition A.7. Namely, $F'$ is a Polish space endowed with
the induced topology. Fix a countable dense set $\Omega'\subseteq
\Omega$, and an open non-empty subset $O$ of $F'$. Then
$\pi_H(F')=\pi_H(\Omega'^{-1}\cdot O)=\Omega'^{-1}\cdot\pi_H( O)$.
It implies that $\pi_H( O')$ is not meager for every non-empty open
subset $O'$ of $\Omega_G$.

Assume the map $\pi_H$ is not ess. category preserving. Then $\pi_H
(B)$ is meager for some set $B\in \Omega_G$ of the second category.
Take a collection $C_i$, $i\in \mathbb{N}$, of closed nowhere-dense
sets such that $\pi_H (B)\subseteq\cup_i C_i$. Then a closed set
$\pi^{-1}_H (C_{i_0})$ is not meager for some $i_0$. It implies that
there is a non-empty open set $O^*\subseteq \pi^{-1}_H (C_{i_0})$.
We have a contradiction and Proposition 7.1 is proved.
\end{proof}
 How much can we extend a typical $H$-action? Obviously, $\pi_H$ is
 still ess. category preserving if we replace $G$ with any free
 product, say $G\ast G_1$, since the topological Fubini theorem
 applies.  Besides, ess. category preserving maps
$\pi^G_H$ (i.e. $\pi_H$ defined on $\Omega_G$) are closed under
taking inductive limits in the following sense.

\begin{proposition} Let $H\leq G_1\ldots\leq G_n\leq \ldots $ be a
countable chain of countable groups $H, G_i, i\in \mathbb{N}$.
Suppose all $\pi^{G_i}_H$ are ess. category preserving; then
$\pi^{\lim_\rightarrow G_i}_H$ is ess. category preserving as well.
\end{proposition}
\begin{proof}
Denote
\[
B=\bigcap_i\pi^{G_i}_H (F_{G_i}), \ B_i=F_{G_i}\cap\bigcap_{j>i}
\pi^{G_j}_{G_i} (F_{G_j}).
\]
Every $H$-action in $B$ can be extended to some element of $B_1$,
every $G_1$-action in $B_1$ can be extended to some element of
$B_2$, and so on. It means that we obtained some subset of
 $\Omega_{\lim_\rightarrow G_i}$, say $C$, satisfying $\pi_H(C)=B$.
 The reader will easily prove that $C\subseteq F'_{\lim_\rightarrow
 G_i}$. To conclude the proof, it remains to apply Proposition 7.1.
\end{proof}
\begin{proposition} Let $G$ be any countable non-abelian group, $H$
its normal abelian subgroup containing an element $h_0$ of infinite
order. Then $\pi^G_H$ is  not ess. category preserving.
\end{proposition}
\begin{proof} Assume $H$ is included in the center of $G$. Since,
for a typical $H$-action $T$, the centralizer $C(T)$ is abelian, any
extension of $T$ to a $G$-action must be abelian. Therefore its
orbit under conjugacies is not dense in $\Omega_G$. So, $\pi^G_H
(F')$ is meager.

If $H$ is not in the center, then we can choose $h_1, h_2 \in H$,
$g\in G$ such that $g^{-1}h_1g=h_2$, $ h_1\neq h_2$. If, in
addition, $g^{-1}h_0^kg=h_0^k$ for some $k\neq 0$, then a typical
$H$-action $T$ can not be extended to any $G$-action. Indeed,
otherwise take some such extension $T$, then $T_{h_i}, T_g \in
CL\{T^n_{kh_0}: n\in \mathbb{Z}\}$ because of Theorem 6.2. Hence,
$T_{h_2}=T_g^{-1}T_{h_1}T_g=T_{h_1}$, and we have a contradiction
with the freeness of a typical $H$-action.

Assume $g^{-1}h_0g=h_3$, $ h_0^k\neq h_3^k$ for any $k\neq 0$. Pick
a positive $k$ such that $H'=<h_0^k, h_3^k>$ is a free abelian
subgroup. It is well known that for a typical transformation all its
powers are not isomorphic to each other (see, for example [18]). It
follows that a typical $H'$-action does not admit any isomorphism
between $T_{h_0^k}$ and $T_{h_3^k}$ if $H'$ is cyclic. The case $H'$
has two generators is analogous. By Theorem 6.2, we conclude that
there is no isomorphism between $T_{h_0^k}$ and $T_{h_3^k}$ for a
typical $H$-action $T$. Hence $T$ can not be extended to any
$G$-action. This completes the proof.
\end{proof}
It is natural to look at possible analogs of Theorems 6.2 and 6.7 if
we omit the restriction to be abelian for $H$ and $G$. Let us list a
series of examples somewhat indicating why there is currently no
complete answer even for $H=\mathbb{Z}$.

 1. The most closed to abelian groups are finite extensions of them.
 Obviously, for every positive integer $n$
\[
G_n={\rm gr}\langle t_1,\ldots , t_n, s; (\forall i,j)
[t_it_j=t_jt_i \& st_is^{-1}=t^{-1}_i]\rangle
\]
contains a free abelian subgroup $H=<t_1,\ldots , t_n, s^2>$ of
index $2$. Arguing as in Remark 6.4, we see that
$\pi_{H'}^{G_n}(\Omega_{G_n})$ is meager for every infinite subgroup
$H'\leq H$.

Besides, it was proved in [4] that for a typical action $T$ of
\[
G_n={\rm gr}\langle t,s; (\forall i,j) [t_i, t_j]=1=t_0\ldots
t_{n-1}\rangle,
\]
where $t_i=s^its^{-i}$, $n>1$, the transformation $T_{s^n}$ has
homogeneous spectrum of multiplicity $n$ (see slightly modified
versions of $G_n$ with the same property in [11], [29]). The groups
$G_n$ are also finite extensions of abelian. Since a typical
transformation has simple spectrum, $\pi_{<s^k>}^{G_n}$ is not ess.
category preserving for every $k\neq 0$. On the other hand, every
transformation $T_s$ can be extended to a $G_n$-action if we put the
identity as  $T_t$. So, $\pi_{<s^k>}^{G_n}(\Omega_{G_n})$ is typical
for every $k\neq 0$.

 2. Our technique, working well for abelian groups, was also
based on approximations by finite actions. Since there exist
countable groups with no non-trivial irreducible finitely
dimensional unitary representations, it may happen that some
non-trivial countable groups do not admit finite actions except an
action by identities. However all the above can (in)directly work
for some groups that are not even virtually abelian. Take, for
example,
\[
G={\rm gr}\langle t_1,s; t_1s=st_1^2\rangle.
\]
 It is well known (see, for example [5], [18] ) that a typical
 transformation is spectrally disjoint to its square. Thus
$\pi^{G}_{<t_1>} (\Omega_{G})$ is meager. Besides, take
\[
G^*={\rm gr}\langle t_1,s,t_2; t_is=st_i^2, i=1,2 \&
t_1t_2=t_2t_1\rangle.
\]
We claim that $\pi^{G^*}_{G}$ is ess. category preserving. Indeed,
first note that $G^*$ being solvable is an amenable group. Observe
that, for a typical $G$-action $T$, $T_{t_1}$ has rank $1$ and is
$rigid$ (i.e. $T_{t_1}^{k_j}$ tends to the identity for some
sequence ${k_j}$) (see [5]). It follows that the centralizer
\[
C\{T_{t_1}\}=CL \{ T_{nt_1} : n\in \mathbb{Z}\}
\]
is uncountable and has no isolated points. Moreover, by the standard
topological arguments, it can be then shown that  the monothetic
group $C\{T_{t_1}\}$ can not be covered by the countable union of
closed sets $T_{jt_1}B_i$, $j\in \mathbb{Z}, i\in \mathbb{N}$, where
$B_i$ consists of all elements $S$ of the centralizer satisfying
$S^i$ is the identity. Pick an element, say $S$, which is not
covered. Then $<T_{t_1}, S>$ form a free $\mathbb{Z}^2$-action. It
easily implies that we extended $T$ to a free $G^*$-action by
putting $T_{t_2}=S$. Applying Proposition 7.1, we conclude that
 $\pi^{G^*}_{G}$ is ess. category preserving.

3. Take a free product of cyclic groups of order $2$ and $3$, i.e.
\[
G={\rm gr}\langle t,s; t^2, s^3\rangle.
\]
Every transformation is a composition two periodic transformations,
say $S_2$ and $S_3$. Moreover, with no loss of the generality, $S_i$
can be chosen satisfying $S^i_i$ ($i=2,3$) are the identity (see
[29]). It implies that $\pi_{<ts>}^G(\Omega_G)$ is typical.

4. Recall that for an inclusion $H\leq G$ of countable, discrete
groups one says that the pair $(G, H)$ has \textbf{relative
property} $\mathbf{(T)}$ if there exists $\delta > 0$ and a finite
$B\subset G$ such that if $\pi$ is a unitary representation of $G$
in a Hilbert space and $f$ is a unit vector satisfying
\[ (\forall g \in B) ||\pi(g)(f)- f|| < \delta,
 \]

then there exists a vector $f_0$ such that $\pi(g)(f_0)= f_0$ for
any $g\in H$. One says that $G$ has \textbf{property} $\mathbf{(T)}$
and is called Kazhdan, if the pair $(G, G)$ has relative property
$(T)$.

Take the pair $H\leq G$ with relative property $(T)$, where $H$ is
not Kazhdan. For a typical $G$-action $T$, $T|_H$ is not ergodic,
and a typical $H$-action is ergodic (see [20]). It implies that
 $\pi_H (F')$ is meager.

One may ask another natural question about how many extensions for a
typical $H$-action we could expect. Arguing in a more or less
standard way, it can be easily answered for countable abelian groups
$G$ (and $H$). Let us only stress that there appear new effects
(i.e. not the continuum or nothing). For example, let $G=H\oplus
\mathbb{Z}/3\mathbb{Z}$, $H=\oplus^\infty\mathbb{Z}/2\mathbb{Z}$. By
Theorem 4.5, for a typical $H$-action $T$ the centralizer of $T$
consists of elements of order $2$. Therefore $T$ has the only
trivial extension to a $G$-action.

It would be interesting to see any differences in the
classifications of pairs $(H\leq G)$ according to Theorems 6.2, 6.7,
6.8 if we take representations of groups by unitary operators of the
separable infinite-dimensional Hilbert space or by isometries of the
universal Urysohn metric space. Besides, it is expected that the
classification of pairs will be essentially different for actions by
measure preserving homeomorphisms of the Cantor set.

It would be also interesting to see if there exists one more
characteristic dynamic property for $\pi_H$ in the above sense.

\bibliographystyle{amsplain}

\end{document}